\documentclass[11pt]{amsart}
\usepackage{amscd}
\usepackage{amsmath}
\usepackage{amsxtra}
\usepackage{amsfonts}
\usepackage{amssymb}

\newtheorem{theorem}{Theorem}[section]
\newtheorem{corollary}[theorem]{Corollary}
\newtheorem{lemma}[theorem]{Lemma}
\newtheorem{proposition}[theorem]{Proposition}

\theoremstyle{definition}
\newtheorem{definition}[theorem]{Definition}
\newtheorem{remark}[theorem]{Remark}

\newtheorem{example}[theorem]{Example}
\theoremstyle{remark}

\renewcommand{\theclaim}{\textup{\theclaim}}

\newtheorem*{acknowledgements}{Acknowledgements}

\numberwithin{equation}{section}

\def\Unitary{\operatorname*{Unitary}}

\def\openone

{\mathchoice

{\hbox{\upshape \small1\kern-3.3pt\normalsize1}}

{\hbox{\upshape \small1\kern-3.3pt\normalsize1}}

{\hbox{\upshape \tiny1\kern-2.3pt\SMALL1}}

{\hbox{\upshape \Tiny1\kern-2pt\tiny1}}}

\makeatletter

\newbox\ipbox

\newcommand{\ip}[2]{\left\langle #1\, , \,#2\right\rangle}
\newcommand{\diracb}[1]{\left\langle #1\mathrel{\mathchoice

{\setbox\ipbox=\hbox{$\displaystyle \left\langle\mathstrut
#1\right.$}

\vrule height\ht\ipbox width0.25pt depth\dp\ipbox}

{\setbox\ipbox=\hbox{$\textstyle \left\langle\mathstrut
#1\right.$}

\vrule height\ht\ipbox width0.25pt depth\dp\ipbox}

{\setbox\ipbox=\hbox{$\scriptstyle \left\langle\mathstrut
#1\right.$}

\vrule height\ht\ipbox width0.25pt depth\dp\ipbox}

{\setbox\ipbox=\hbox{$\scriptscriptstyle \left\langle\mathstrut
#1\right.$}

\vrule height\ht\ipbox width0.25pt depth\dp\ipbox}

}\right. }

\newcommand{\dirack}[1]{\left. \mathrel{\mathchoice

{\setbox\ipbox=\hbox{$\displaystyle \left.\mathstrut
#1\right\rangle$}

\vrule height\ht\ipbox width0.25pt depth\dp\ipbox}

{\setbox\ipbox=\hbox{$\textstyle \left.\mathstrut
#1\right\rangle$}

\vrule height\ht\ipbox width0.25pt depth\dp\ipbox}

{\setbox\ipbox=\hbox{$\scriptstyle \left.\mathstrut
#1\right\rangle$}

\vrule height\ht\ipbox width0.25pt depth\dp\ipbox}

{\setbox\ipbox=\hbox{$\scriptscriptstyle \left.\mathstrut
#1\right\rangle$}

\vrule height\ht\ipbox width0.25pt depth\dp\ipbox}

} #1\right\rangle}

\newcommand{\cj}[1]{\overline{#1}}

\newcommand{\bz}{\mathbb{Z}}

\newcommand{\B}{\mathcal{B}}
\newcommand{\br}{\mathbb{R}}
\def\W{\mathcal{W}}
\newcommand{\bc}{\mathbb{C}}
\newcommand{\bt}{\mathbb{T}}
\newcommand{\bn}{\mathbb{N}}

\def\blfootnote{\xdef\@thefnmark{}\@footnotetext}

\renewcommand{\mod}{\operatorname{mod}}

\newcommand{\bu}{\mathbf u}
\input xy
\xyoption{all}
\usepackage{amssymb}

\newcommand{\supp}[1]{\text{supp} (#1)}

\newcommand{\Span}{\overline{\operatorname*{span}}}

\def\R{\mathcal{R}}

\def\C{\mathcal{C}}

\def\H{\mathcal{H}}

\def\-{^{-1}}
\def\B{\mathcal{B}}

\def\O{\mathcal{O}}
\def\K{\mathcal{K}}

\def\ty{\emptyset}

\def\W{\mathcal{W}}
\def\Orbit{\operatorname*{Orbit}}

\begin{document}

\title[Atomic representations of Cuntz algebras]{Atomic representations of Cuntz algebras}
\author{Dorin Ervin Dutkay}

\address{[Dorin Ervin Dutkay] University of Central Florida\\
	Department of Mathematics\\
	4000 Central Florida Blvd.\\
	P.O. Box 161364\\
	Orlando, FL 32816-1364\\
U.S.A.\\} \email{Dorin.Dutkay@ucf.edu}

\author{John Haussermann}

\address{[John Haussermann] University of Central Florida\\
	Department of Mathematics\\
	4000 Central Florida Blvd.\\
	P.O. Box 161364\\
	Orlando, FL 32816-1364\\
U.S.A.\\} \email{jhaussermann@knights.ucf.edu}

\author{Palle E.T. Jorgensen}
\address{[Palle E.T. Jorgensen]University of Iowa\\
Department of Mathematics\\
14 MacLean Hall\\
Iowa City, IA 52242-1419\\}\email{palle-jorgensen@uiowa.edu}

\thanks{} 

\begin{abstract}
       To a representation of $\O_N$ (the Cuntz algebra with $N$ generators) we associate a projection valued measure and we study the case when this measure has atoms. The main technical tool are the spaces invariant for all the operators $S_i^*$. We classify the purely atomic representations and find when such representations are permutative. Applications include: wavelet representations, representations generated by finitely correlated states, representation associated to Hadamard triples (or fractal spectral measures) and representations associated to generalized Walsh bases. 
\end{abstract}
\maketitle \tableofcontents

\section{Introduction}

 The Cuntz algebras are indexed by an integer $N > 1$, where $N$ is the number of generators. As a $C^*$-algebra (denoted $\O_N$), it is defined by generators and relations, and it is known to be a simple, purely infinite $C^*$-algebra, \cite{Cu77}. Further its $K$-groups are known. But its irreducible representations are highly subtle. To appreciate the importance of the study of representations of $\O_N$, recall that to specify a representation of $\O_N$ amounts to identifying a system of isometries in a Hilbert space $\H$, with orthogonal ranges adding up to $\H$. But such orthogonal splitting in Hilbert space may be continued iteratively, and as a result, one gets links between $\O_N$-representation theory to such neighboring areas as symbolic dynamics; and to filters used in signal processing, corresponding to a system of $N$ uncorrelated frequency bands.

Returning to the subtleties of the representations of $\O_N$, and their equivalence classes, it is known, for fixed $N$, that the set of equivalence classes of irreducible representations  of $\O_N$,  does not admit a Borel cross section; i.e., the equivalence classes, under unitary equivalence, does not admit a parameterization in the measurable Borel category. (Intuitively, they defy classification.) Nonetheless, special families of inequivalent representations have been found, and they have a multitude of applications, both to mathematical physics \cite{BrJo02} , to the study of wavelets \cite{DuJo08a, DuJo07a, Jor06, Jor01,MR1285568,MR1158756,MR2966145,MR2277210,MR1352420,MR2563494,MR2362879,MR2318495,MR2121531,MR2560855,MR2531317,MR2412296,MR2251300,MR2129642,MR1980913} , to harmonic analysis \cite{Str89, DuJo9a, DuJo07b}, to the study of fractals as iterated function systems \cite{DuJo06a, DuJo11a}; and to the study of $\operatorname*{End}(B(\H))$ (= endomorphisms)  where $\H$ is a fixed Hilbert space.  Hence it is of interest to identify both discrete and continuous series of representations of $\O_N$; as they arise in such applications.

      We begin with a systematic study of $\operatorname*{Rep}(\O_N, \H)$ where $\H$ is a fixed Hilbert space. In section 2, we compute, starting with a fixed representation of $\O_N$ (the Cuntz algebra with $N$ generators), an associated spectral resolution of a maximal abelian algebra computed from the symbolic presentation of $\O_N$. This takes the form of a projection valued measure $P$ on the Borel subsets of the Cantor group $\K_N$ on a finite alphabet, see e.g., \cite{Rud90, Kat04}.  The relevance of these projection valued measures includes decompositions of $L^2(\br)$ with respect to wavelet packets, as well as to general and canonical decomposition of representations of $\O_N$, see e.g., \cite{DuJo12}.

      The paper is structured as follows: in section 2, we introduce the projection valued measure $P$ associated to a representation of the Cuntz algebra and derive some of its general properties. The atoms are the infinite words $\omega$ for which $P(\{\omega\})\neq 0$. In section 3 we study subspaces which are invariant for the adjoints of the Cuntz isometries $(S_i^*)_{i\in\bz_N}$. The main result of this section is Theorem \ref{th2.6} which shows that for a finite dimensional, cyclic invariant subspace $M$, there can be no non-periodic atoms and all cyclic atoms must be contained in $M$. In section 4 we study purely atomic representations, that is representations for which the projection valued measure $P$ is supported on a set of atoms. A particular case of purely atomic representations are the permutative representations introduced in \cite{BrJo99}. This is proved in Theorem \ref{th1.2}. We also give necessary and sufficient conditions for a purely atomic representation to be permutative (Theorems \ref{th1.4}, \ref{th1.5}). In Theorem \ref{th1.3} we show that under some mild assumptions, any $(S_i^*)$-invariant subspace $M$ as above, must contain one of the vectors in the permutative basis. In Theorem \ref{th3.11} and Corollary \ref{cor4.18} we show that any finite dimensional $(S_i^*)$-invariant subspace in a purely atomic representation must contain some elements in a cyclic atom and nothing from the non-periodic ones. In section 5 we study various classes of examples. In Theorem \ref{th4.3} we show that wavelet representations cannot have atoms except in some very special circumstances. In section 5.2 we turn to representation generated by finitely correlated states, that is representations obtained by a dilation of a finite dimensional space. We give conditions under which these representations are purely atomic or permutative. Section 5.3 refers to representations associated to Hadamard triples or fractal spectral measures. These representations are always permutative and in Theorem \ref{opr8.3}, we use our results to give a detailed description of the structure of such representations. A particular case (Example \ref{ex8.4}) involves the classical Fourier series on the interval.  In section 5.4 we turn to representations of $\O_N$ associated to generalized Walsh bases; these representations are also permutative and they are all equivalent (Theorem \ref{th9.1}).
      
      {\bf A notation:} We will sometimes identify a closed subspace with the orthogonal projection onto it; so, for example we will use the same letter $M$ for an invariant subspace and for the projection onto it, or we use $P(\omega)$ for the projection onto an atom and also for the subspace which is the range of the projection. Other times we will use the notation $P_M$ to indicate the projection onto the subspace $M$, and we use $P(\omega)\H$ for the range of the projection $P(\omega)$, where $\H$ is the ambient Hilbert space.

\section{Representations of $\O_N$}

Even though, the classification problem for representations of $\O_N$, defies parametrization of the equivalence classes if we insist on covering all representations \cite{BJKW00}, it was found that for practical problems, one may limit the focus to special representations. These various subclasses (dictated by a host of applications) have nonetheless found to admit very computable invariants, up to unitary equivalence. For the case of wavelets and signal processing (quadrature mirror filters), see \cite{BrJo02, BEJ00, Jor01, Jor06, MR2277210}, fractals \cite{Hut81, DJ06, DuJo08a, MR2966145, MR2563494}, harmonic analysis \cite{DuJo9a, DHS12}  , affine geometry \cite{DuJo07b}, statistics \cite{DJ06, MR1158756} , super-selection sectors and deformations  in physics \cite{MR1158756, MR2412296, MR1980913, MR2251300, MR2531317, MR2560855,  Cu77, MR2362879,  MR1285568, MR2318495}, ergodic theory \cite{MR2560855, BrJo99, BJO04},  combinatorics and graph theory \cite{MR2121531, MR2318495}. This is by no means an exhaustive list.

    The purpose of the present section is to offer a geometric framework for approaches to classification. Much of this is motivated by analogous results in the area of unitary representations of groups, but the case of the Cuntz algebras $\O_N$ offers a host of subtleties not present is the traditional treatments of decomposition theory.

\begin{definition}\label{def2.0}
Let $N\geq 2$. The Cuntz algebra $\O_N$ is the $C^*$-algebra generated by some isometries $(S_i)_{i\in\bz_N}$ satisfying the {\it Cuntz relations}
\begin{equation}
S_i^*S_j=\delta_{ij}I,\quad (i,j\in\bz_N),\quad \sum_{i\in\bz_N}S_iS_i^*=I.
\label{eq2.0.1}
\end{equation}
\end{definition}
\begin{definition}\label{def0.1}
Fix an integer $N\geq 2$. Let $(S_i)_{i=0}^{N-1}$ be a representation of the Cuntz algebra $\O_N$ on a Hilbert space $\H$. Let $\bz_N:=\{0,1,\dots,N-1\}$. We will call elements in $\bz_N^k$ {\it words of length $k$}. 
We denote by $\K=\K_N=\bz_N^{\bn}$, the set of all infinite words.
Given two finite words $\alpha=\alpha_1\dots\alpha_n$, $ \beta=\beta_1\dots\beta_m$, we denote by $\alpha\beta$ the concatenation of the two words $\alpha\beta=\alpha_1\dots\alpha_n\beta_1\dots\beta_m$. Similarly for the case when $\beta$ is infinite. Given a word $\omega=\omega_1\omega_2\dots$ and $k$ an non-negative integer smaller than its length we denote by 
$$\omega|k:=\omega_1\dots\omega_k.$$

For a finite word $I=i_1\dots i_n$ we denote by 
$$S_I:=S_{i_1}\dots S_{i_n}.$$

\end{definition}

\begin{definition}\label{def1.3}
An infinite word $\omega$ in $\K$ is called {\it periodic} if $\sigma^k(\omega)=\sigma^{k'}(\omega)$ for some non-negative integers $k\neq k'$. An infinite word $\omega$ is called cyclic if there exists $p>0$ such that $\sigma^p(\omega)=\omega$. In this case $\omega$ is an infinite repetition of the word $\omega_1\dots\omega_p$ and we write $\omega=\underline{\omega_1\dots\omega_p}$. A subset $\Omega$ of $\K$ is called {\it aperiodic} if it contains no periodic points.
\end{definition}

\begin{definition}\label{defcyc}
Let $(S_i)_{i=0}^{N-1}$ be a representation of the Cuntz algebra $\O_N$ on a Hilbert space $\H$. A set $\mathcal M$ of vectors in $\H$ is called {\it cyclic} for the representation if 
$$\Span\{S_IS_J^*v : v\in \mathcal M, I, J\mbox{ finite words }\}=\H.$$

\end{definition}

\begin{definition}\label{defcantor}
Fix an integer $N\geq 2$.  The {\it Cantor group on $N$ letters} is 
$$\K=\K_N=\prod_{1}^\infty\bz_N=\{(\omega_1\omega_2\dots) : \omega_i\in \bz_N\mbox{ for all }i=1,\dots\},$$
an infinite Cartesian product.

The elements of $\K$ are infinite words.
On the Cantor group, we consider the product topology. We denote by $\B(\K)$ the sigma-algebra of Borel subsets of $\K$. 

Denote by $\sigma$ the shift on $\K$, $\sigma(\omega_1\omega_2\dots)=\omega_2\omega_3\dots$. Define the inverse branches of $\sigma$: for $i\in\bz_N$, $\sigma_i(\omega_1\omega_2\dots)=i\omega_1\omega_2\dots$.

For a finite word $I=i_1\dots i_k\in \bz_N^k$, we define the corresponding {\it cylinder set }

\begin{equation}
\C(I)=\{\omega\in \K : \omega_1=i_1,\dots,\omega_k=i_k\}.
\label{eqcantor0}
\end{equation}

The Pontryagin dual group of $\K$ is 
$$\widehat \K=\{(\xi_1\xi_2\dots\xi_p00\dots)\, : \xi_1,\dots,\xi_p\in \bz_N, p\in\bn\}.$$
The Fourier duality is given by 
\begin{equation}
\ip{\omega}{\xi}=\prod_{k=1}^\infty \ip{\omega_k}{\xi_k}=\prod_{k=1}^\infty e^{2\pi i\frac{\omega_k\xi_k}N},\quad(\omega\in\K, \xi\in\widehat\K).
\label{eqcantor1}
\end{equation}
Note that only finitely many terms in this product are not equal to 1. 

For $\xi=\xi_1\dots\xi_n0\dots\in\widehat \K$ and $I=i_1\dots i_n\in\bz_N^n$ we define 
$$\ip{I}{\xi}=\ip{I0\dots}{\xi}.$$

Define the dual of the shift $\widehat \sigma$ on $\widehat K$ by $\ip{\omega}{\widehat\sigma(\xi)}=\ip{\sigma(\omega)}{\xi}$ for all $\omega\in\K$, $\xi\in\widehat \K$, so
$$\widehat\sigma(\xi_1\dots\xi_p0\dots)=(0\xi_1\dots\xi_p0\dots).$$
\end{definition}

\begin{theorem}\label{th0.2}
Let $(S_i)_{i\in\bz_N}$ be a representation of the Cuntz algebra $\O_N$ on a Hilbert space $\H$. For every finite word $I$, define the projection
\begin{equation}
P(\C(I))=S_IS_I^*.
\label{eq0.2.1}
\end{equation}
Then $P$ extends to a projection-valued measure on $\B(\K)$.

For $\xi\in\widehat\K$, $\xi=\xi_1\dots\xi_n0\dots$, define the operator 
\begin{equation}
U(\xi)=\sum_{I\in\bz_N^n}\ip{I}{\xi}S_IS_I^*.
\label{eq0.2.2}
\end{equation}
Then $U$ defines a unitary representation of $\widehat \K$ on $\H$. Also the projection-valued measure $P$ is the spectral decomposition ( as defined by the Stone-Naimark-Ambrose-Godement theorem \cite{MR1043174}) of the unitary representation $U$, i.e.,
\begin{equation}
U(\xi)=\int_{\K}\ip{\omega}{\xi}P(d\omega),\quad(\xi\in\widehat\K).
\label{eq0.2.3}
\end{equation}
For $i\in \bz_N$ and $A\in \B(\K)$
\begin{equation}
S_iP(A)S_i^*=P(\sigma_i(A)).
\label{eq0.2.3.1}
\end{equation}

\begin{equation}
\sum_{i\in\bz_N}S_iP(\sigma_i^{-1}(A))S_i^*=P(A),\quad(A\in\B(\K)).
\label{eq0.2.3.2}
\end{equation}

Define $\alpha:B(\H)\rightarrow B(\H)$ by
\begin{equation}
\alpha(A)=\sum_{i=0}^{N-1}S_iAS_i^*,\quad (A\in \B(\H)).
\label{eq0.2.4}
\end{equation}
Then $\alpha$ is an endomorphism of $B(\H)$ and 
\begin{equation}
\alpha(P(A))=P(\sigma^{-1}(A)),\quad (A\in \B(\K)).
\label{eq0.2.5}
\end{equation}
For every bounded measurable function $f$ on $\K$ 
\begin{equation}
\alpha\left(\int_\K f(\omega)dP(\omega)\right)=\int_\K f(\sigma(\omega))dP(\omega).
\label{eq0.2.6}
\end{equation}
Also 
\begin{equation}
\alpha(U(\xi))=U(\widehat\sigma(\xi)),\quad(\xi\in\K).
\label{eq0.2.7}
\end{equation}

\end{theorem}

\begin{proof}
For the definition of $P$, we use the Kolmogorov extension theorem, and we just have to check the consistency conditions. If $I=i_1\dots i_n$ and $I=i_1'\dots i_n'$ are different words then the cylinders $\C(I)$ and $\C(I')$ are disjoint. The Cuntz relations imply that the projections $P(\C(I))=S_IS_I^*$ and $P(\C(I'))=S_{I'}S_{I'}^*$ are orthogonal. 

Also for a finite word $I$, the cylinder $\C(I)$ is the disjoint union of the cylinders $\C(Ii)$ for $i\in\bz_N$. The Cuntz relations imply that 
$$\sum_{i\in\bz_N} P(\C(Ii))=\sum_{i\in\bz_N}S_{Ii}S_{Ii}^*=S_IS_I^*=P(\C(I)).$$
Thus, the consistency relations are satisfied and therefore $P$ extends to a projection-valued measure on $\B(\K)$.

For the operators $U$, we check first that they are consistently defined: if $\xi=\xi_1\dots\xi_p0\dots$ in $\widehat \K$ then $\xi_{p+1}=0$ and we have
$$\sum_{I\in\bz_N^{p+1}}\ip{I}{\xi}S_IS_I^*=\sum_{J\in\bz_N^p,i\in\bz_N}\ip{Ji}{\xi}S_JS_iS_i^*S_J^*=\sum_{J}\sum_i\ip{J}{\xi}S_JS_iS_i^*S_J^*=\sum_{J}\ip{J}{\xi}S_JS_J^*.$$

Then, take $\xi,\xi'\in \widehat \K$ and take $p$ larger than the length of both $\xi$ and $\xi'$ we have, using the Cuntz relations:
$$U(\xi)U(\xi')=\sum_{I,I'\in \bz_N^p}\ip{I}{\xi}\ip{I'}{\xi'}S_IS_I^*S_{I'}S_{I'}^*=\sum_{I}\ip{I}{\xi+\xi'}S_IS_I^*=U(\xi+\xi').$$

To verify \eqref{eq0.2.3}, take $\xi=\xi_1\dots\xi_p0\dots$ in $\widehat \K$ and note that the character $\omega\mapsto \ip{\omega}{\xi}$ is constant $\ip{I}{\xi}$ on the cylinders $\C(I)$ with $I\in\bz_N^p$. Therefore 
$$\int_K\ip{\omega}{\xi}dP(\omega)=\sum_{I\in\bz_N^p}\ip{I}{\xi}P(\C(I))=\sum_{I\in\bz_N^p}\ip{I}{\xi}S_IS_I^*=U(\xi).$$

Equation \eqref{eq0.2.3.1} can be checked easily on cylinders on account of the Cuntz relations; and, since the sigma-algebra $\B(\K)$ is generated as a monotone class by cylinders, it extends to all Borel sets. Same for \eqref{eq0.2.3.2}. Then \eqref{eq0.2.5} follows from \eqref{eq0.2.3.1}; and \eqref{eq0.2.6} follows from \eqref{eq0.2.5} by the usual approximation argument and \eqref{eq0.2.7} follows from \eqref{eq0.2.6} by taking $f(\omega)=\ip{\omega}{\xi}$. 
\end{proof}

\begin{definition}\label{def2.3}
We call the projection valued measure $P$ defined in Theorem \ref{th0.2}, the {\it projection valued measure associated to the representation} $(S_i)_{i\in\bz_N}$. 

We will use also the following notations, for an infinite word $\omega\in \K_N$ and for a finite word $I$:
$$P(\omega)=P(\{\omega\}),\quad P(I)=P(\C(I))=S_IS_I^*.$$
\end{definition}

\begin{corollary}\label{cor2.3}
For any Borel set $A$ in $\K_N$ and $i\in\bz_N$:
\begin{equation}
S_iP(\sigma_i^{-1}(A))=P(A)S_i,\quad S_iP(A)=P(\sigma^{-1}(A))S_i.
\label{eqco2.2.1}
\end{equation}
If $A$ is a Borel subset of $\K$ with $A=\sigma^{-1}(A)$ then $P(A)$ commutes with the representation of $\O_N$. 
\end{corollary}

\begin{proof}
From \eqref{eq0.2.3.1} we have 
$$S_iP(\sigma_i^{-1}(A))=P(\sigma_i(\sigma_i^{-1}(A)))S_i=P(A\cap \C(i))S_i=P(A)P(\C(i))S_i=P(A)S_iS_i^*S_i=P(A)S_i.$$
Then 
$$P(\sigma^{-1}(A))S_i=S_iP(\sigma_i^{-1}(\sigma^{-1}(A)))=S_iP((\sigma\circ\sigma_i)^{-1}(A))=S_iP(A).$$
If $A=\sigma^{-1}(A)$ then, \eqref{eqco2.2.1} shows that $P(A)$ commutes with the isometries and, taking the adjoint, it commutes also with their adjoints. 
\end{proof}

\begin{corollary}\label{lem2.7}
Let $(S_i)_{i\in\bz_N}$ be a representation of $\O_N$ on a Hilbert space $\H$ and let $P$ be the associated projection valued measure. For $\psi\in\H$, set 
$$\mu_\psi(B)=\|P(B)\psi\|^2=\ip{P(B)\psi}{\psi},\quad(B\in\B(\K)).$$
Then 
$$\mu_\psi(B)=\sum_{i\in\bz_N}\mu_{S_i^*\psi}(\sigma_i^{-1}(B)),\quad(B\in\B(\K)).$$
\end{corollary}

\begin{proof}

Follows directly from \eqref{eq0.2.3.2}.
\end{proof}

\begin{proposition}\label{pr2.8}
Let $(S_i)_{i\in \bz_N}$ and $(\tilde S_i)_{i\in\bz_N}$ be two representations of $\O_N$ on the Hilbert spaces $\H$ and $\tilde \H$ and let $X:\H\rightarrow \tilde \H$ be an intertwining operator, i.e., 
\begin{equation}
\tilde S_iX=XS_i,\quad \tilde S_i^*X=XS_i^*,\quad (i\in\bz_N).
\label{eqpr2.8.0}
\end{equation}

 Let $P$, $\tilde P$ be the associated projection valued measures. Then $X$ intertwines $P$ and $\tilde P$, i.e., 
$$XP(A)=\tilde P(A)X,\quad (A\in \B(\K_N)).$$

If $A,B$ are two disjoint Borel sets in $\K_N$ then $\tilde P(A)XP(B)=0$. In particular if $w,w'$ are two different finite words of the same length and if $\omega, \omega'$ are two different infinite words, then 
\begin{equation}
\tilde P(w)XP(w')=0,\quad \tilde P(\omega)XP(\omega)=0.
\label{eqpr2.8.1}
\end{equation}

Also
\begin{equation}
\tilde S_wXS_w^*=\tilde P(w)X P(w).
\label{eqpr2.8.2}
\end{equation}

If $P$ and $\tilde P$ are supported on disjoint sets , i.e., there exist disjoint Borel sets $A,B$ in $\K_N$ such that $\tilde P(A)=I$, $P(B)=I$, then the representations are disjoint.
\end{proposition}

\begin{proof}
If $X$ intertwines the representations, then $X$ intertwines the two unitary groups $U(\xi)$, $\tilde U(\xi)$, $\xi\in\widehat \K_N$. Then, $X$ must intertwine their two spectral resolutions which are $P$ and $\tilde P$. 
If $A$, $B$ are disjoint sets, then 
$$\tilde P(A)X P(B)=XP(A)P(B)=0.$$
Equations \eqref{eqpr2.8.1} and \eqref{eqpr2.8.2} follow immediately. 

If $P$ and $\tilde P$ are supported on disjoint sets, we have 
$$X=\tilde P(A)XP(B)=0,$$
so the representations are disjoint.

\end{proof}

\section{$S_i^*$-invariant spaces}
 In the study of decomposition theory for representations $(S_i)_{i\in\bz_N}$ of the Cuntz algebra $\O_N$ acting on a Hilbert space $\H$, care must be taken in the distinction between closed subspaces in $\H$ which are invariant under the system $S_i$, vs the tuple of adjoint operators $S_i^*$. Indeed for reasons of applications to quadrature mirror filters in signal processing \cite{Jor06} one finds that invariance under the system $(S_i^*)_{i\in\bz_N}$ is the right starting point. This will be justified below in a geometric framework.

\begin{proposition}\label{pr2.4}
Let $(S_i)_{i\in\bz_N}$ be a representation of $\O_N$ on a Hilbert space $\H$. Let $M$ be a closed subspace of $\H$ and let $P_M$ be the projection onto $M$. The following statements are equivalent:
\begin{enumerate}
	\item $M$ is invariant for the operators $S_i^*$, i.e., $P_MS_i^*P_M=S_i^*P_M$, $i\in\bz_N$. 
	\item With the endomorphism $\alpha$ as in \eqref{eq0.2.4}, $P_M\leq \alpha(P_M)$. 
\end{enumerate}
In this case $\{\alpha^n(P_M)\}_{n\in]\bn}$ is an increasing sequence of projections; the projection $Q=\vee_n\alpha^n(P_M)$ commutes with the representation; if $Q_n=\alpha^n(P_M)-\alpha^{n-1}(P_M)$ for $n\geq1$, and $Q_0=P_M$, then the projections $Q_n$, $n\geq 0$ are mutually orthogonal and 
$$\oplus_{n=0}^\infty Q_n=Q.$$

The following relations hold:
\begin{equation}
S_i^*\alpha^n(P_M)=\alpha^{n-1}(P_M)S_i^*,\quad S_i\alpha^{n-1}(P_M)=\alpha^n(P_M)S_i,\quad (n\geq 1, i\in\bz_n),
\label{eqpr2.4.1}
\end{equation}

\begin{equation}
S_i^*Q_1=P_MS_i^*(I-P_M),\quad             S_i^*Q_{n+1}=Q_{n}S_i^*,\quad S_iQ_n=Q_{n+1}S_i,\quad (n\geq1, i\in\bz_N).
\label{eqpr2.4.2}
\end{equation}

Also, the following are equivalent
\begin{enumerate}
	\item $M$ is cyclic for the representation.
	\item $Q=\vee_n\alpha^n(P_M)=I$.
	\item $\oplus_{n\geq0}Q_n=I$.
\end{enumerate}

\end{proposition}

\begin{proof}
For the first equivalence, if $M$ is invariant, take $v\in M$, then, for all $i\in\bz_N$, $S_i^*v\in M$ so $P_MS_i^*v=S_i^*v$ so 
$$\alpha(P_M)v=\sum_{i\in \bz_N}S_iP_MS_i^*v=\sum_{i\in\bz_N}S_iS_i^*v=v.$$

Conversely, if $P_M\leq \alpha(P_M)$ then take $v\in M$ then 
$$v=\alpha(P_M)v=\sum_{i\in\bz_N}S_iP_MS_i^*v,$$
so $S_j^*v=P_MS_j^*v$, that is $S_j^*v\in M$ for all $j\in\bz_N$.

Since $\alpha$ is an endomorphism $\alpha(P_M)\leq \alpha(\alpha(P_M))$ so, by induction, the sequence $\{\alpha^n(P_M)\}$ is increasing. Clearly, this implies that the projections $Q_n$ are mutually orthogonal. Equations \eqref{eqpr2.4.1}, \eqref{eqpr2.4.2} follow directly from the Cuntz relations and the invariance of $M$.

For the last equivalence, it is clear that (ii) is equivalent to (iii). 

Note that 
\begin{equation}
\alpha^n(P_M)\H=\Span\{S_Iv : |I|=n, v\in M\}.
\label{eqpr2.4.3}
\end{equation}
Indeed, if $v=\sum_{|I|=n}S_IP_MS_I^*v$, then, since $P_MS_I^*v\in M$, the vector $v$ is the span. Conversely, take $S_Jv$ with $v\in M$ and $|J|=n$ then $\alpha^n(P_M)S_Jv=S_Jv$ so $S_Jv$ is in $\alpha^n(P_M)\H$.

The equivalence of (i) and (ii) follows from this.

\end{proof}

\begin{theorem}\label{th2.6}
Let $(S_i)_{i\in\bz_N}$ be a representation of $\O_N$ on a Hilbert space $\H$ and let $P$ be the associated projection valued measure. Suppose $M$ is a finite dimensional, cyclic invariant subspace for $S_i^*$, $i\in\bz_N$. Then for any non-periodic word $\omega\in \K$, $P(\omega)=0$. For any cyclic word $\omega$, $P(\omega)\H\subset M$. 
\end{theorem}

\begin{proof}

We need some lemmas:

\begin{lemma}\label{lem2.7.1}
For every infinite word $\omega$ in $\K$:
 \begin{equation}
 S_iP(\omega)S_i^*=P(i\omega),\mbox{ so }S_{\omega_1}^* P(\omega)S_{\omega_1}=P(\sigma(\omega)),\mbox{ and }S_{i}^* P(\omega)S_{i}=0\mbox{ for }i\neq \omega_1.
\label{eq1.4.1}
\end{equation}
Therefore
\begin{equation}
S_iP(\omega)\H=P(i\omega)\H,\quad S_{\omega_1}^*P(\omega)\H=P(\sigma(\omega))\H,\quad S_i^*P(\omega)\H=0\mbox{ for }i\neq \omega_1.
\label{eq1.4.2}
\end{equation}
Also, $S_i$ is unitary between $P(\omega)\H$ and $P(i\omega)\H$ and $S_i^*$ is unitary between $P(i\omega)\H$ and $P(\omega)\H$. 
\end{lemma}

\begin{proof}

Everything follows easily from \eqref{eq0.2.3.1}.
\end{proof}

\begin{lemma}\label{lem2.8}
If $M$ is a cyclic invariant subspace for $S_i^*$, $i\in\bz_N$, and $P_M$ is the projection onto $M$, then, for every $v\in \H$ and every $\epsilon>0$ there exists $n_\epsilon\in\bn$ such that for any finite word $I$ with $|I|\geq n_\epsilon$, we have $\|P_MS_I^*v-S_I^*v\|<\epsilon$. 

\end{lemma}

\begin{proof}

Write $v=\sum_{n=0}^\infty Q_nv$ as in Proposition \ref{pr2.4}. Then $\sum_n\|Q_nv\|^2=\|v\|^2<\infty$. So, there exists $n_\epsilon$ such that $\sum_{n\geq n_\epsilon}\|Q_nv\|^2<\epsilon$. 

Now take a finite word $I$ with $|I|\geq n_\epsilon$. For $n\leq |I|$, using \eqref{eqpr2.4.2}, we have $S_I^*Q_nv\in M$. 

For $n>|I|$ we have $S_I^*Q_nv=Q_{n-|I|}S_I^*v$. Then 
$$\left\|P_MS_I^*v-S_I^*v\right\|^2=\left\|\sum_{n\leq |I|}(P_MS_I^*Q_nv-S_I^*Q_nv)+\sum_{n> |I|}(P_MS_I^*Q_nv-S_I^*Q_nv)\right\|^2$$
$$=\left\|\sum_{n>|I|}(0-Q_{n-|I|}S_I^*v)\right\|^2=\sum_{n>|I|}\|Q_{n-|I|}S_I^*v\|^2=\sum_{n>|I|}\|S_I^*Q_nv\|^2\leq \sum_{n>|I|}\|Q_nv\|^2<\epsilon.$$

\end{proof}

We return to the proof of the theorem. Let $\omega$ be a non-periodic word. Suppose $P(\omega)\neq 0$ and let $v\in P(\omega)\H$, $\|v\|=1$. Then, with Lemma \ref{lem2.7.1}, the sequence $S_{\omega|k}^*v$, $k\geq 0$ is orthonormal.

With Lemma \ref{lem2.8}, $P_MS_{\omega|k}^*v-S_{\omega|K}^*v$ converges to 0.  Since $M$ is finite dimensional there exists a subsequence such that $P_MS_{\omega|n_k}^*v$ converges to some $u\in M$. Then also $S_{\omega|n_k}^*v$ converges to $u$. But, since these vectors are orthonormal, the distance between them is always $\sqrt{2}$ so the subsequence cannot converge.

Now take $\omega$ a cyclic word $\omega=\underline I$ and assume $P(\omega)\neq 0$ and suppose there is a $v\in P(\omega)\H$, $v\neq 0$ such that $v\perp (P(\omega)\H\cap M)$.

Since $M$ is invariant and with Lemma \ref{lem2.7.1}, $S_I^*(M\cap P(\omega)\H)\subset M\cap P(\omega)\H$. Since $M$ is finite dimensional and $S_I^*$ is unitary on $P(\omega)\H$, we actually get $S_I^*(M\cap P(\omega)\H)=M\cap P(\omega)\H$. 
Let $L$ be the orthogonal complement of $M\cap P(\omega)\H$ in $P(\omega)\H$. Then we also have $S_I^*L=L$, again, since $S_I^*$ is unitary on $P(\omega)$. Then, as $v\in L$ we get that $S_{I^n}^*v\in L$ and $\|S_{I^n}^*v\|=1$, for all $n\geq 0$.

With Lemma \ref{lem2.8}, we have $P_MS_{I^n}^*v-S_{I^n}^*v\rightarrow 0$. Since $M$ is finite dimensional there is subsequence such that $P_MS_{I^{n_k}}^*v$ converges to some $u\in M$. Also, we have 
$$P_{M\cap P(\omega)\H}P_MS_{I^n}^*v=P_MP_{M\cap P(\omega)\H}S_{I^n}^*v=0,$$
so $P_MS_{I^n}^*v\perp (M\cap P(\omega)\H)$. Then $u\perp (M \cap P(\omega)\H)$. 

We have also $S_{I^{n_k}}^*v\rightarrow u$ so $u\in P(\omega)\H$ and therefore $u\in M\cap P(\omega)\H$. But this means that $u=0$ because $u$ belongs to $M\cap P(\omega)$ and its complement; and also, $\|u\|=\lim \|S_{I^{n_k}}^*v\|=\|v\|\neq 0$, a contradiction. 

\end{proof}

\begin{example}\label{ex3.5}
If the invariant space in Theorem \ref{th2.6} is not finite dimensional, then the result is not necessarily true. Let $\W(\underline 0)$ be the set of infinite words that end in $\underline0$. Define a representation of the Cuntz algebra $\O_2$ on $l^2(\W(\underline0))$ by 
$$S_i\delta_{w\underline0}=\delta_{iw\underline0},\quad i=0,1,$$
where $\delta_{w\underline0}$ is the canonical basis for $l^2(\W(\underline0))$. 

Consider the unitary shift operator $U$ on $l^2(\bz)$, $U\delta_n=\delta_{n+1}$, $n\in\bz$. 

Define the operators $\tilde S_i$ on $l^2(\bz)\otimes l^2(\W(\underline0))$ by 
$$\tilde S_0(\delta_n\otimes\delta_{\underline0})=\delta_{n+1}\otimes\delta_{\underline0},\quad \tilde S_0(\delta_n\otimes\delta_{w\underline0})=\delta_{n}\otimes\delta_{0w\underline0},\quad(w\neq00\dots0)$$$$\tilde S_1(\delta_n\otimes \delta_{w\underline 0})=\delta_n\otimes\delta_{1w\underline0}.$$
Thus
$$\tilde S_0=U\otimes P_0+I\otimes (S_0P_0^\perp ),\quad \tilde S_1=I\otimes S_1,$$
where $P_0$ is the projection onto the span of $\delta_{\underline 0}$. Note that $P_0$ commutes with $S_0$ and $S_0^*$ and so the same is true for $P_0^\perp$. 

It is easy to check that $(\tilde S_i)_{i=0,1}$ defines a representation of $\O_2$. 

Let 
$$M=\Span\{\delta_n\otimes \delta_{w\underline 0} : n\leq 0, w\mbox{ finite word}\}.$$
Then $M$ is invariant for $S_0^*,S_1^*$ and cyclic for the representation. 

However, if $P,\tilde P$ are the associated projection valued measure, then 
$$\tilde P(\underline 0)=\lim \tilde S_0^n(\tilde S_0^*)^n=\lim_n I\otimes P_0+ I\otimes S_0^n(S_0^*)^nP_0^\perp=\lim_n I\otimes P(\underbrace{0\dots 0}_{\mbox{$n$ times}})
=P_{l^2(\bz)\otimes\Span\{\delta_{\underline0}\}}$$ is not contained in $P_M$. 
\end{example}

    The theorem of Wold states that if $S$ is a single isometry acting in a fixed Hilbert space $\H$, then $\H$ splits up as a sum of two orthogonal subspaces, canonical, in the sense that $S$ will restrict to the first subspace to be a shift operator, while $S$ restricts to a unitary operator in the second closed subspace. Now a representation of  $\O_N$ is a system of isometries (subject to \eqref{eq2.0.1}), and so the Wold decomposition applied to one $S_i$ will then have to satisfy consistency properties with respect to the other isometries. The following Remark is fleshing out these consistency conditions. To ease notation below we will denote the respective closed subspaces in the Hilbert space $\H$ with subscripts ``shift'' and ``unitary'' respectively. For further details regarding the Wold decomposition and $\O_N$-representations, see \cite{BJO04}.

\begin{remark}\label{rem3.5} The atoms corresponding to cycle points $\omega=\underline I$ coincide with the unitary part of the Wold decomposition for the isometry $S_I$.
Let $I=i_1\dots i_p$ be some finite word and let $\Unitary(S_I)$ be the unitary part in the Wold decomposition of $S_I$, i.e.,
$$\Unitary(S_I)=\bigcap_{n\geq1}S_I^n\H=\bigcap_{n\geq 0}S_I^n(S_I^*)^n\H=\{x\in\H : \|S_I^nx\|=\|x\|,\mbox{ for all }n\geq0\}.$$

But since $$S_I^n(S_I^n)^*=P(\underbrace{II\dots I}_{\mbox{$n$ times}})$$ it follows that 
$$\Unitary(S_I)=P(\underline I)\H.$$

If $I,I'$ are two finite words such that $\underline I\neq \underline I'$ then it follows that the unitary parts of the Wold decompositions are orthogonal
$$\Unitary(S_I)\perp \Unitary(S_{I'}).$$

Also, if we define the subspace 
$$M(\underline I)=\oplus_{k=1}^n\Unitary(S_{i_{k}\dots i_pi_{1}\dots i_{k-1}}),$$
then 
$$M(\underline I)=\oplus_{k=0}^{n-1}P(\sigma^k(\underline I))\H,$$
and therefore we see, with Lemma \ref{lem2.7.1} that $M(\underline I)$ is invariant for $S_i^*$, $i\in\bz_N$.

With \eqref{eqpr2.4.3} we see that 
$$\alpha^n(P_{M(\underline I)})=\oplus \{P(w\underline I): |w|\leq n\}$$
and, with $Q(\underline I)=\vee_n\alpha^n(P_{M(\underline I)})$ we have, that $Q(\underline I)$ is invariant for the representation and  
$$Q(\underline I)=\oplus\{P(\omega) : \omega\in \Orbit(\underline I)\},$$
where $\Orbit(\underline I)$ is the orbit of $\underline I$ under $\sigma,\sigma^{-1}$, 
$$\operatorname*{Orbit}(\underline I)=\cup_{k,l\geq0}\sigma^{-k}\sigma^l(\underline I).$$

If we consider the spaces $\Unitary(S_i)$, $i\in\bz_N$ and then construct the space $\H_{unit}:=\oplus_{i\in\bz_N}Q(\underline i)$ then $\H_{unit}$ is invariant for the representation and if $\H_{shift}$ is its complement, then $\H_{shift}$ is also invariant for the representation and moreover the restrictions of the isometries $S_i$ to $\H_{shift}$ are pure isometries.

Similarly, if we take all finite words $I$ and construct the space $\H_{U}=\oplus\{M(\underline I): I\mbox{ finite word}\}$ and let $\H_{S}$ be its orthogonal complement, then $\H_{S}$ is invariant for the representation and all isometries $S_I$ are pure shifts on $\H_S$.
\end{remark}

\section{Atomic and permutative representations}
  In \cite{BrJo99}, the authors introduced a family of $\O_N$-representations which has since proved to be especially amenable to applications in discrete mathematics (see e.g., \cite{MR2121531, MR2318495}), as well as in physics \cite{MR2412296, MR1980913, MR2251300, MR2531317, MR2560855}. Below we study the relationship between these representations of $\O_N$ to the family of purely atomic representations introduced here.

   While there is no known classification of all representations of $\O_N$, the authors of \cite{BrJo99} identified a more amenable subclass of $\O_N$ representations, named permutative representations (Definition \ref{def1.1}, named thus because they permute the elements in some special orthogonal basis in the Hilbert space carrying the representation. Combining this with the Cuntz relations \eqref{eq2.0.1}, one gets a partition of the index set $I$ for the special ONB. What makes the permutative representations amenable to computations (classification and applications) is the existence of an encoding mapping $E$ from the index set $I$ into the space $\K_N$ of infinite words in the alphabet $\bz_N$. Properties of permutative $\O_N$ representations may thus be calculated from the associated encoding mappings, and vice versa. From the permutative property of such a representation, we get a mapping system on the index set $I$ for the ONB, $\sigma$ and $(\sigma_l)$, branches of the inverse for $\sigma$ (in $I$), so endomorphisms in $I$. Now the encoding mapping $E$ goes from $I$ to $\K_N$, and  $\K_N$ in turn carries its own shift maps, i.e., shift to the left, and shift to the right with insertion of a letter on the first place (see Definition \eqref{defcantor}. A key property of the encoding mapping $E$ (Proposition \ref{pr1.2}) is that it intertwines the respective mapping systems.   As a result the range of $E$, the subset $E(I)$ (in $\K_N$) is a non-sofic subshift (see e.g., \cite{Th04}); actually $E(I)$ is doubly invariant.

\begin{definition}\label{def1.2}
Let $(S_i)_{i\in\bz_N}$ be a representation of the Cuntz algebra $\O_N$ on a Hilbert space $\H$. We say that the representation is {\it purely atomic } if the associated projection-valued measure $P$ is purely atomic, i.e., there exist a subset $\Omega$ of $\K$ such that 
$$\oplus_{\omega\in\Omega}P(\omega)=I.$$
We also say that the representation is supported on $\Omega$. For a purely atomic representation, we call the set 
$$\supp P=\{\omega\in\K :P(\omega)\neq0\}$$
 {\it the support} of $P$.
\end{definition}

\begin{proposition}\label{pra.1}
Let $(S_i)_{i\in\bz_N}$ be a representation of $\O_N$ on the Hilbert space $\H$, with projection valued measure $P$. Define
\begin{equation}
A(\H):=\{\omega\in\K : P(\omega)\neq0\},\quad \H_{atomic}:=\oplus_{\omega\in A(\H)}P(\omega)\H,\quad\H_{cont}:=\H\ominus\H_{atomic}.
\label{eqa.1.1}
\end{equation}
Then the set $A(\omega)$ is invariant for the maps $\sigma$ and $\sigma_i$, $i\in\bz_N$; the spaces $\H_{atomic}$ and $\H_{cont}$ are invariant for the representation. The subsrepresentation of $\O_N$ on $\H_{atomic}$ is purely atomic, and the subrepresentation on $\H_{cont}$ has no atoms. 
\end{proposition}

\begin{proof}
From Lemma \ref{lem2.7.1}, we see that $P(\omega)\H\neq 0$ implies that $P(\sigma_i(\omega))\H\neq0$ and $P(\sigma(\omega))\H\neq0$ so $A(\H)$ invariant under $\sigma_i$ and $\sigma$. Then, Corollary \ref{cor2.3}, implies that $\H_{atomic}$ is invariant for the representation and so $\H_{cont}$ is too. 

The projection valued measure associated to the subrepresentation on $\H_{atomic}$ is $P_a(A)=P(A)P_{\H_{atomic}}$ for all $A\in\B(\K)$. Therefore 
$$\oplus_{\omega\in A(\H)}P_a(\omega)=I_{\H_{atomic}},$$
so, this subrepresentation is purely atomic. 

If, by contradiction, the subrepresentation on $\H_{cont}$ has an atom, $\omega$, then $P(\omega)P_{\H_{cont}}\neq 0$ so $P(\omega)\neq0$. But then $\omega\in A(\H)$ so $P(\omega)\H\subset \H_{atomic}$ and therefore $P(\omega)P_{\H_{cont}}=0$, a contradiction. 

\end{proof}

\begin{proposition}\label{pr3.6}
Let $(S_i)_{i\in\bz_N}$ be a purely atomic representation of the Cuntz algebra $\O_N$ and let $P$ be the associated projection valued measure, with support $\Omega$. Decompose $\Omega$ into disjoint orbits under $\sigma$, $\sigma^{-1}$, i.e., $\operatorname*{Orbit}(\omega)=\cup_{k,l\geq0}\sigma^{-k}\sigma^l(\omega)$:
$$\Omega=\bigcup_\omega \operatorname*{Orbit}(\omega).$$
Then the representation splits into a direct sum of representations, more precisely, for each $\omega$ the projection $P(\operatorname*{Orbit}(\omega))$ is invariant for the representation and 
$$\oplus_{\omega}P(\operatorname*{Orbit}(\omega))=I.$$

\end{proposition}

\begin{proof}
The fact that $P(\Orbit(\omega))$ is invariant for the representation is a consequence of Corollary \ref{cor2.3}. Everything else follows from this. 
\end{proof}
\begin{proposition}\label{pr2.10}
Let $(S_i)_{i\in\bz_N}$ be an irreducible representation of $\O_N$. Suppose the representation has an atom $\omega$. Then the representation is purely atomic and supported on $\Orbit(\omega)$.
\end{proposition}

\begin{proof}
Since $\sigma^{-1}(\Orbit(\omega))=\Orbit(\omega)$, it follows that $P(\Orbit(\omega))\neq0$ is invariant for the representation, and since this is irreducible, we get that $P(\Orbit(\omega))=I$ which implies that the representation is purely atomic. 

\end{proof}

\begin{proposition}\label{pr2.11}
Let $(S_i)_{i\in\bz_N}$ and $(\tilde S_i)_{i\in\bz_N}$ be two purely atomic representations of $\O_N$, with associated projection valued measures $P$ and $\tilde P$. Assume that the representations are supported on $\Orbit(\omega)$ and $\Orbit(\omega')$ respectively for two infinite words $\omega$ and $\omega'$. 

\begin{enumerate}
	\item If $\Orbit(\omega)\neq \Orbit(\omega')$ then the representations are disjoint. 
	\item If $\omega=\omega'$ and $\omega$ is not periodic, there is a linear isometric isomorphism between intertwiners $X:\H\rightarrow \tilde \H$ and operators $Y:P(\omega)\rightarrow \tilde P(\omega)$ defined by $Y=\tilde P(\omega)XP(\omega)$. 
	\item If $\omega=\omega'$ and $\omega=\underline I$ is a cycle, then there is a linear isometric isomorphism between intertwiners $X:\H\rightarrow \tilde \H$ and operators $Y:P(\omega)\rightarrow \tilde P(\omega)$ with the property 
	\begin{equation}
YS_I=\tilde S_IY\mbox{ on }P(\omega),
\label{eqpr2.11.1}
\end{equation}
	defined by $Y=\tilde P(\omega)XP(\omega)$. 
\end{enumerate}
\end{proposition}

\begin{proof}
Distinct orbits are disjoint so (i) follows from Proposition \ref{pr2.8}.

For (ii), we have to check that each operator $Y:P(\omega)\rightarrow \tilde P(\omega)$ induces an intertwiner $X$ such that $\tilde P(\omega)XP(\omega)=Y$. We need a lemma

\begin{lemma}\label{lem2.12}
Let $a,a'$ be two finite words, $k,k'\geq0$ and $\omega$ and infinite non-periodic word. If $a\sigma^k(\omega)=a'\sigma^{k'}(\omega)$ then $S_aS_{\omega|k}^*=S_{a'}S_{\omega|k'}^*$ on $P(\omega)$. 

\end{lemma}

\begin{proof}

So assume $a\sigma^k(\omega)=a'\sigma^{k'}(\omega)$ and in addition that the length of $a'$ is bigger than or equal to the length of $a$. Then $a'$ must be of the form $aa_0$ and $\sigma^k(\omega)=a_0\sigma^{k'}(\omega)$. Let $m$ be the length of $a_0$. We have $\sigma^{k+m}(\omega)=\sigma^{k'}(\omega)$. Since $\omega$ is non-periodic, we have $k'=k+m$. And therefore $a_0=\omega_{k+1}\dots\omega_{k+m}$. 

Then, on $P(\omega)$, we have (since the isometries are unitary between atoms, by Lemma \ref{lem2.7.1}): 
$$S_{a'}S_{\omega|k'}^*=S_aS_{a_0}S_{\omega_{k+m}}^*\dots S_{\omega_{k+1}}^*S_{\omega|k}^*=S_aS_{\omega|k}^*.$$
\end{proof}

Every point in the orbit of $\omega$ is of the form $a\sigma^k(\omega)$. Define $X$ from $P(a\sigma^k(\omega))$ to $\tilde P(a\sigma^k(\omega))$ by 
$$X=\tilde S_a\tilde S_{\omega|k}^*YS_{\omega|k}S_a^*.$$

Lemma \ref{lem2.12} shows that this operator is well defined, in the sense that it does not depend on the witting $a\sigma^k(\omega)$ for a point in the orbit of $\omega$. We check that $X$ intertwines.

Let $i \in \bz_N$, and $\phi \in P(a\sigma^k(\omega))$. Then $S_i \phi \in P(ia\sigma^k(\omega))$, so
$$ X S_i \phi =\tilde S_{ia}\tilde S_{\omega|k}^*YS_{\omega|k}S_{ia}^* S_i \phi = \tilde S_i \tilde S_a\tilde S_{\omega|k}^*YS_{\omega|k}S_a^* S_i^* S_i \phi = \tilde S_i X \phi .$$
Now consider $S_i^* \phi$. If $i$ is not the first letter of $a$, this is zero. Similarly, if $i$ is not the first letter of $a$ then $\tilde S_i^* \tilde S_a =0$, so it is enough to check that $\tilde S_i^* X = X S_i^*$ when $i$ is the first letter of $a$. In this case, $a=ia'$ and  $S_i^* \phi \in P(a' \sigma^k(\omega))$, so we have
$$ X S_i^* \phi = \tilde S_{a'}\tilde S_{\omega|k}^*YS_{\omega|k}S_{a'}^* S_i^* \phi  = \tilde S_{a'}\tilde S_{\omega|k}^*YS_{\omega|k}S_a^*  \phi = \tilde S_i^* X \phi .$$
We now have that $X$ is an intertwiner.

For (iii) , the proof is similar to that for (ii), but note that in this case 

$$\tilde S_I\tilde P(\omega)XP(\omega)=\tilde P(I\omega)\tilde S_IXP(\omega)=\tilde P(\omega)XS_IP(\omega)=\tilde P(\omega)XP(\omega)S_I.$$
Conversely, if $Y$ satisfies this intertwining relation, then $X$ defined as above, in the proof of (ii), intertwines the representations. Note that, since $S_I$ and $\tilde S_I$ are unitary on $P(\omega)$ and $\tilde P(\omega)$ respectively, by the Fuglede-Putnam theorem \cite{MR0077905}, the operator $Y$ also intertwines $S_I^*$ and $\tilde S_I^*$.

\end{proof}

\begin{corollary}\label{cor2.13}
Let $(S_i)_{i\in\bz_N}$ be a purely atomic representation of $\O_N$ with associated projection valued measure $P$. Suppose $P$ is supported on $\Orbit(\omega)$.
Then the representation is irreducible if and only if $\dim P(\omega)=1$. 
\end{corollary}

\begin{proof}
If $\omega$ is non-periodic, then the commutant is isomorphic to the space of bounded operators on $P(\omega)$, so it is one-dimensional iff $P(\omega)$ is. 
If $\omega=\underline I$ is a cycle and $\dim P(\omega)>1$ then there are non-trivial operators which commute with the unitary $S_I$ on $P(\omega)$. So the commutant is trivial iff $\dim P(\omega)=1$. 
\end{proof}

\begin{corollary}\label{cor2.14}
Let $(S_i)_{i\in\bz_N}$, $(\tilde S_i)_{i\in\bz_N}$ be two purely atomic representations of $\O_N$ on the Hilbert spaces $\H$, $\tilde H$ with projection valued measure $P$ and $\tilde P$. The two representations are unitarily equivalent if and only if the following conditions are satisfied:
\begin{enumerate}
	\item $\supp P=\supp {\tilde P}=:\Omega$;
	\item For every $\omega\in\Omega$, $\dim P(\omega)=\dim \tilde P(\omega)$;
	\item If $\omega\in \Omega$ is a cycle $\omega=\underline I$, then there exists a unitary map $U_\omega:P(\omega)\rightarrow \tilde P(\omega)$ such that $U_\omega S_I=\tilde S_IU_\omega$ on $P(\omega)$. 
\end{enumerate}

\end{corollary}

\begin{proof}
If $U:\H\rightarrow\tilde \H$ is a unitary operator that intertwines the representations then 
$$\tilde P(A)=UP(A)U^*$$
for any Borel subset of $\K$. (i)--(iii) follow from directly from this.

Conversely, decompose $\Omega$ into orbits
$$I_\H=\oplus P(\Orbit(\omega)),\quad I_{\tilde\H}=\oplus \tilde P(\Orbit(\omega)).$$
We will define the intertwining unitary on each component; therefore we can assume $\Omega=\Orbit(\omega)$ for some $\omega\in\K$. 

If $\omega$ is not periodic, then take any unitary $Y:P(\omega)\rightarrow \tilde P(\omega)$, which exists by (ii), and, by Proposition \ref{pr2.11} (ii), we get an intertwining unitary $X$ from $\H$ to $\tilde\H$ such that $\tilde P(\omega)XP(\omega)=Y$. 

If $\omega$ is a cycle then we use Proposition \ref{pr2.11}(iii) with $Y=U_\omega$ as in (iii) and we obtain an intertwining unitary $X:\H\rightarrow\tilde\H$. 

\end{proof}

We recall some fact about permutative representations, form \cite{BrJo99}.

\begin{definition}\label{def1.1}
A representation $(S_l)_{i\in\bz_N}$ of the Cuntz algebra $\mathcal O_N$ on a Hilbert space $\H$ is called {\it permutative} if there exists an orthonormal basis $\{e_i : i\in I\}$ such that for all $i\in I$, $j\in \bz_N$, the vector $S_je_i$ is an element of this basis. This defines the {\it branching maps} $\sigma_j:I\rightarrow I$ by

\begin{equation}
S_je_i=e_{\sigma_j(i)},\quad(j\in\bz_N,i\in I).
\label{eq1.1.1}
\end{equation}
\end{definition}

\begin{proposition}\label{pr1.2}\cite{BrJo99}
The maps $\sigma_l$ are one-to-one and 
\begin{equation}
\bigcup_{j\in\bz_N}\sigma_j(I)=I,
\label{eq1.2.1}
\end{equation}
\begin{equation}
\sigma_j(I)\cap\sigma_{j'}(I)=\ty,\quad(j\neq j').
\label{eq1.2.2}
\end{equation}

Also
\begin{equation}
S_j^*e_{\sigma_j(i)}=e_i,\quad S_{j'}^*e_{\sigma_j(i)}=0\mbox{ for }j\neq j'.
\label{eq1.2.3}
\end{equation}
For each $i\in I$ there exits uniquely $j_0\in L$ and $i_1\in I$ such that $i=\sigma_j(i_1)$. We denote $\sigma(i)=i_1$, so $\sigma_j(\sigma(i))=i$ for all $i\in\sigma_j(I)$ and $\sigma(\sigma_j(i))=i$ for all $i\in I$. The map $\sigma:I\rightarrow I$ is $N$-to-1. (See also Remark \ref{rem4.20}).

Define the {\it coding map} $E: I\rightarrow \K_N $ by
\begin{equation}
E(i)=j_0j_1\dots, \mbox{ where }\sigma^k(i)\in\sigma_{j_k}(I),\mbox{ for all } k\in\bn.
\label{eq1.2.4}
\end{equation}

 Then: the set $E(I)$ is invariant for the maps $\sigma_l$ and for $\sigma$; the following intertwining relations hold:  

\begin{equation}
\sigma_l\circ E=E\circ\sigma_l,\quad \sigma\circ E=E\circ\sigma.
\label{eq1.2.5}
\end{equation}

Thus, if the coding map is one-to-one, then the set $I$ can be identified with an invariant subset of words $E(I)$ in $\K_N$ and the maps $\sigma_l$ and $\sigma$ can be identified with the corresponding shifts on this set.
\end{proposition}

\begin{remark}\label{rem4.20}
Note that we have to systems of maps $\sigma, \sigma_l$, one on the index set $I$ and one on the infinite words $\K_N$. We use the same letters for both of them to simplify the notation, but the reader should be aware of the difference. 
\end{remark}

\begin{theorem}\label{th1.2}
Let $(S_i)_{i=0}^{N-1}$ be a permutative representation of the Cuntz algebra $\O_N$ on a Hilbert space $\H$ with orthonormal basis $\{e_\lambda : \lambda\in\Lambda\}$ and encoding mapping $E$. Then the representation is purely atomic and supported on $E(\Lambda)$. Moreover, for $\omega\in E(\Lambda)$, $P(\omega)$ is the projection onto the closed span of the vectors $e_\lambda$, $\lambda\in E^{-1}(\omega)$. 
\end{theorem}

\begin{proof}
It is enough to prove the last statement, because then
$$\oplus_{\omega\in E(\Lambda)}P(\omega)=\oplus_{\omega\in E(\Lambda)}\oplus_{\lambda\in E^{-1}(\omega)}|e_\lambda\rangle\langle e_\lambda|=I,$$
since $\{e_\lambda\}$ is an orthonormal basis.

For the last statement, take $\omega\in E(\Lambda)$, $\omega=E(\lambda)$. Note that, by the definition of $E(\lambda)$, we have that for all $n$ : $e_\lambda\in S_{\omega_1}\dots S_{\omega_n}\H=P(\C(\omega|n))\H$. The intersection of the cylinders $\C(\omega|n)$ is $\{\omega\}$ so $e_\lambda\in P(\omega)\H$. 

Now, take $\lambda'\not\in E^{-1}(\omega)$, so $E(\lambda')=\omega'\neq \omega$. Then $e_\lambda\in P(\omega')\H \perp P(\omega)\H$. These relations imply that $P(\omega)$ is indeed the projection onto the span of $e_\lambda$, $\lambda\in E^{-1}(\omega)$.

\end{proof}

\begin{theorem}\label{th1.4}
Let $(S_i)_{i\in\bz_N}$ be a purely atomic representation of the Cuntz algebra $\O_N$. If the representation is supported on an aperiodic set, then the representation is permutative. Moreover the representation can be decomposed into a direct sum of permutative representations with injective coding map.
\end{theorem}

\begin{proof}
Let $\Omega$ be the support of $P$. By Proposition \ref{pr3.6}, we can decompose $\Omega$ into orbits and each orbit will give a subrepresentation. Thus, by taking direct sums, it is enough to consider the case when $\Omega$ is the orbit of some non-periodic word $\omega$ in $\K$. 

 Let $\{e_i\}_{i\in I}$ be some orthonormal basis in $P(\omega)\H$.

Every point in $\Omega$ is of the form $a\sigma^k(\omega)$ for some finite word $a$ and some non-negative integer $k$.

Define the orthonormal basis in $P(\{a\sigma^k(\omega)\})\H$ by $S_aS_{\omega|k}^*e_i$, $i\in I$. The relations in \eqref{eq1.4.2} show that this is indeed an orthonormal basis for $P(\{a\sigma^k(\omega)\})\H$. Lemma \ref{lem2.12} shows that this does not depend on the choice of the way we write $a\sigma^k(\omega)$. 

Since $P$ is supported on $\Omega$, we have $\oplus_{a,k} P(\{a\sigma^k(\omega)\})=I$, so the union of these orthonormal bases form an orthonormal basis for $\H$. We have to check only that the Cuntz isometries map the basis into itself. But this is clear from the definition. 

To decompose the representation into a direct sum of permutative representations with injective coding maps, consider, for each vector $e_i\in O$, the space 
$$\H_i=\cj{\operatorname*{span}}\{S_aS_{\omega|k}^*e_i : k\geq 0, a\mbox{ finite word}\}.$$
It is easy to see that the direct sum of the spaces $\H_i$ is $\H$, they are invariant for the representation, and the restriction of the representation has a permutative orthonormal basis $\{S_aS_{\omega|k}^*e_i: k\geq 0, a\mbox{ finite word}\}$. The coding map is $E(S_aS_{\sigma^k(\omega)}^*e_i)=a\sigma^k(\omega)$ and it is injective. 
\end{proof}

\begin{theorem}\label{th1.5}
Let $(S_i)_{i\in\bz_N}$ be a purely atomic representation of the Cuntz algebra $\O_N$ on a Hilbert space $\H$. Suppose the representation is supported on a set $\Omega$ that contains only periodic points. The following statements are equivalent:
\begin{enumerate}
	\item The representation is permutative. 
	\item For every cycle $\omega=\underline{\omega_1\dots\omega_p}$ in the support $\Omega$, there exists an orthonormal basis $O_\omega=\{e_i : i\in I_\omega\}$ for $P(\omega)\H$ such that $S_{\omega_1}\dots S_{\omega_p}O_\omega=O_\omega$. 
\end{enumerate}

\end{theorem}

\begin{proof}
Assume (i). Let $\omega=\underline{\omega_1\dots\omega_p}$ in $\Omega$. 
From \eqref{eq1.4.2} it follows that 
\begin{equation}
S_{\omega_1}\dots S_{\omega_{p}}P(\omega)\H=P(\omega)\H. 
\label{eq1.5.1}
\end{equation}
Let $O=\{e_\lambda : \lambda\in E^{-1}(\omega)\}$. Then (ii) follows from Theorem \ref{th1.2}

Assume now (ii). As in the proof of Theorem \ref{th1.4}, we can assume that the support $\Omega$ consists of a single orbit of a cyclic word $\omega=\underline{\omega_1\dots\omega_p}$. Assume in addition that $p$ is minimal with the property that $\sigma^p(\omega)=\omega$. Then each word in $\Omega$ can be written as $i_1\dots i_n\omega$. Note that if $i_1\dots i_n\omega=j_1\dots j_{n+m}\omega$ then $i_1=j_1,\dots,i_n=j_n$, $m$ is a multiple of $p$, $m=kp$,  and $j_{n+1}\dots j_m$ is a repetition of $\omega_1\dots\omega_p$, $k$ times. 

Using again the relations \eqref{eq1.4.2}, define the orthonormal basis in $P(\{i_1\dots i_n\omega\})\H$ by $S_{i_1}\dots S_{i_n}O_\omega$. Using the previous relations, and the hypotheses in (ii), we see that this does not depend on the choice of the writing $i_1\dots i_n\omega$. And it is clear that the union of all these orthonormal bases for all choices of $i_1\dots i_n\omega$ forms an orthonormal basis for $\H$. 
\end{proof}

\begin{theorem}\label{th1.3}
Let $(S_i)_{i\in \bz_N}$ be a permutative representation of $\mathcal O_N$ with orthonormal basis $\{e_i : i\in I\}$ and assume that the coding map is injective. Suppose $P_M$ is a projection onto a closed subspace $M$ invariant under all $S_l^*$, $l\in \bz_N$. Assume that the following condition is satisfied:
\begin{equation}
\mbox{There exists $\lambda_0\in I$ such that }\|P_Me_{\lambda_0}\|=\max\{\|P_Me_i\| : i\in I\}.
\label{eq1.3.1}
\end{equation}
Then $e_{\lambda_0}$ is in $M$.
\end{theorem}

  \begin{proof}
  Let $(u_i)_{i\in J}$ be an orthonormal basis for $M$. Since the space is invariant under $S_l^*$ there exist constants $\alpha_{ij}^l$ such that
  $$S_l^*u_i=\sum_{j\in J}\alpha_{ji}^lu_j.$$
  Decompose $u_i$ in the basis $\{e_i\}$
  $$u_i=\sum_{\lambda\in I}u_{i,\lambda}e_\lambda.$$
  We have, with \eqref{eq1.2.3} 
  $$S_l^*u_i=\sum_{\lambda'\in I}u_{i,\sigma_l(\lambda')}e_{\lambda'}.$$
  Also
  $$S_l^*u_i=\sum_{j\in J}\alpha_{ji}^l\sum_{\lambda\in I}u_{j,\lambda}e_\lambda=\sum_{\lambda\in I}\sum_{j\in J}\alpha_{ji}^lu_{j,\lambda}e_\lambda.$$
  Therefore
  $$u_{i,\sigma_l(\lambda)}=\sum_{j\in J}\alpha_{ji}^lu_{j,\lambda}.$$
  
  Define the column vector $\bu_\lambda=(u_{j,\lambda})_{j\in J}^T$. Note that
  
  \begin{equation}
  \|\bu_\lambda\|^2=\sum_{j\in J}|\ip{u_j}{e_\lambda}|^2=\sum_{j\in J}|\ip{u_j}{P_Me_{\lambda}}|^2=\|P_Me_\lambda\|^2.
\label{eq1.3.p}
\end{equation}

  Define also the matrices $Y_l=(\alpha_{ij}^l)_{i,j\in J}$. Let $Z_l = Y_l^T$. Thus, we have
  \begin{equation}
  \bu_{\sigma_l(\lambda)}=Z_l\bu_\lambda,\quad(\lambda\in I).
	\label{eq1.3.2}
	\end{equation}
  
  Using the relation $\sum_l S_lS_l^*=I$ and the fact that all the vectors $S_lu_i$, $l\in \bz_N$, $i\in J$ are all orthogonal we get that
  \begin{equation}
\sum_{l\in L}Z_l Z_l^* =I.
\label{eq1.3.3}
\end{equation}

For $\omega=l_0\dots l_{n-1}$ in $\bz_N^n$, denote by $Z_\omega:=Z_{l_0}\dots Z_{l_{n-1}}$.

By induction, for all $n\in\bn$:
\begin{equation}
\sum_{\omega\in \bz_N^n}Z_\omega Z_\omega^*=I.
\label{eq1.3.3'}
\end{equation}

 This implies that 
 \begin{equation}
\sum_{\omega\in \bz_N^n}\|Z_\omega^* v\|^2=\|v\|^2,\quad(v\in P\H).
\label{eq1.3.4}
\end{equation}

In particular $\|Z_l^* \|=\|Z_l\|\leq 1.$

From \eqref{eq1.3.2}, we get that $\|\bu_{\sigma_l(\lambda)}\|\leq \|\bu_\lambda\|$ and, from this, $\|\bu_\lambda\|\leq\|\bu_{\sigma(\lambda)}\|$, for all $\lambda\in I$. Apply this to $\lambda_0$, we get $\|\bu_{\sigma(\lambda_0)}\|=\|\bu_{\lambda_0}\|$. By induction $\|\bu_{\sigma^n(\lambda_0)}\|=\|\bu_{\lambda_0}\|$ for all $n\in\bn$. 

Let $E(\lambda_0)=l_0l_1\dots$ be the coding of $\lambda_0$. For any $n\in\bn$, let $\omega_0=l_0\dots l_{n-1}$, we have 
$$\|\bu_{\lambda_0}\|^2=\|Z_{\omega_0}\bu_{\sigma^n(\lambda_0)}\|^2=|\ip{Z_{\omega_0}^*Z_{\omega_0}\bu_{\sigma^n(\lambda_0)}}{\bu_{\sigma^n(\lambda_0)}} | \leq \|\bu_{\sigma^n(\lambda_0)}\|^2=\|\bu_{\lambda_0}\|^2.$$
Since we have equalities in all inequalities we get that $Z_{\omega_0}^*Z_{\omega_0}\bu_{\sigma^n(\lambda_0)}=\bu_{\sigma^n(\lambda_0)}$. Thus $Z_{\omega_0}^* \bu_{\lambda_0}=\bu_{\sigma^n(\lambda_0)}$. 

With \eqref{eq1.3.3'}, we obtain $Z_{\omega'}^* \bu_{\lambda_0}=0$ for all $\omega'\neq\omega_0$ in $\bz_N^n$, which implies that $\bu_{\lambda_0}\perp Z_{\omega'}v$ for all $\omega'\neq\omega_0$ and $v\in M$.

But then, for $\lambda=l_0'l_1'\dots \neq \lambda_0$, since the coding map is injective there exists $n$ such that $l_0\dots l_{n-1}\neq l_0'\dots l_{n-1}'$, and this means $\bu_{\lambda_0}\perp Z_{l_0'}\dots Z_{l_{n-1}'}\bu_{\sigma^n(\lambda)}=\bu_{\lambda}.$

Thus $\bu_{\lambda_0}\perp \bu_{\lambda}$ for all $\lambda\neq\lambda_0$. 

We can change the orthonormal basis $(u_j)$ is such a way that $\bu_{\lambda_0}$ has the form $(\alpha,0,\dots)^T$, so the only non-zero component is on some fixed position $j_0$. Indeed, just pick some unitary operator $A=(a_{ji})_{j,i\in J}$ such that $A\bu_{\lambda_0}$ has the given form, and define $v_i=\sum_{j\in J}a_{ji}u_j$, for all $i\in J$. Note that the norms $\|\bu_\lambda\|$ are preserved by \eqref{eq1.3.p}. Since $\bu_\lambda\perp\bu_{\lambda_0}$ for $\lambda\neq \lambda_0$, we get that $\bu_\lambda$ is of the form $(0,*,*,\dots)^T$. But then $u_j=\alpha e_{\lambda_0}+\sum_{\lambda\neq \lambda_0} 0\cdot e_\lambda$. This means that $e_{\lambda_0}$ is in $M$.
  
  \end{proof}

  \begin{corollary}\label{cor1.4}
  In the hypothesis of Theorem \ref{th1.3}, if $M$ is finite dimensional then there exists a cycle $\lambda$ in $I$, i.e., $\sigma^n(\lambda)=\lambda$ for some $n\geq1$ and $e_\lambda\in M$. 
  \end{corollary}
  
  \begin{proof}
  Take $u_1,\dots, u_p$ an orthonormal basis for $M$. We use the notations in the proof of Theorem \ref{th1.3}. We have $\|\bu_\lambda\|=\|P_Me_\lambda\|$ and 
  $$\sum_\lambda\|P_Me_\lambda\|^2=\sum_\lambda \|\bu_\lambda\|^2=\sum_\lambda\sum_{j=1}^n|u_{j,\lambda}|^2=\sum_{j=1}^p\sum_\lambda|u_{j,\lambda}|^2=\sum_{j=1}^p\|u_j\|^2=p.$$
  Thus all but finitely many numbers $\|P_Me_\lambda\|$ are close to zero and therefore there is a maximum in \eqref{eq1.3.1}. 
  
  By Theorem \ref{th1.3}, there is some $\lambda$ such that $e_\lambda\in M$. Since $M$ is invariant under all$S_l^*$, we get that $e_{\sigma^n(\lambda)}\in M$ for all $n$. Since the space is finite dimensional, there are $m\neq n$ such that $\sigma^n(\lambda)=\sigma^m(\lambda)$ and the corollary follows.

  \end{proof}

  \begin{remark}\label{rem3.10}
  Theorem \ref{th1.3} and Corollary \ref{cor1.4} are not true if the encoding map is not injective. Here is an example: consider two different symbols $a$ and $b$, define $\Lambda$ to be the set $$\Lambda=\{a,b, wa, wb\mbox{ where $w$ is a word on 0,1 of length $\geq1$, $w\neq\ty$}\}.$$
  The Hilbert space is $l^2(\Lambda)$. Define the isometries $S_0$, $S_1$ by $S_0\delta_a=\delta_b$, $S_0\delta_b=\delta_a$, $S_0\delta_{wa}=\delta_{0wa}$, $S_0\delta_{wb}=\delta_{0wb}$, $S_1\delta_a=\delta_{1a}$, $S_1\delta_b=\delta_{1b}$, $S_1\delta_{wa}=\delta_{1wa}$, $S_1\delta_{wb}=\delta_{1wb}$. 
  
  It is easy to check that this defines a permutative representation of $\O_2$ with basis $\{\delta_\lambda : \lambda\in\Lambda\}$. 
  
  Let $M$ be the 1-dimensional space spanned by $\delta_a+\delta_b$. We have $S_0^*(\delta_a+\delta_b)=\delta_b+\delta_a$ and $S_1^*(\delta_a+\delta_b)=0$. So $M$ is invariant for $S_i^*$, $i=1,2$. However $M$ does not contain any element of the basis. The coding map is {\it not} injective, since the coding of $a$ and $b$ is $\underline 0$. 
  
  \end{remark}

  \begin{theorem}\label{th3.11}
  Let $(S_i)_{i\in\bz_N}$ be a purely atomic representation of $\O_N$ on a Hilbert space $\H$ with projection valued measure $P$. Suppose $M$ is a finite dimensional subspace which is invariant for all the maps $S_i^*$, $i\in\bz_N$. Then there exists a cycle word $\omega$ and a vector $v\neq 0$ in $P(\omega)\H\cap M$. Moreover, if $\operatorname*{Per}$ is the set of periodic points in the support of $P$, then 
  \begin{equation}
  M\subset P(\operatorname*{Per})\H.
\label{eq3.11.1}
\end{equation}

   \end{theorem}
   
   \begin{proof}
   If the representation is decomposed into a direct sum of representations, by projecting $M$ onto the components of the direct sum, we can assume that $M$ is contained in one of these components. We can decompose the representation into distinct orbits. 
   
   Suppose $v\in M\cap P(\omega)\H$ for some non-periodic $\omega$. Then the vectors $S_{\omega|k}^*v\in M$ are orthogonal for $k\geq 0$, by Lemma \ref{lem2.7.1}. Since $M$ is finite dimensional we get that $v=0$. This proves \eqref{eq3.11.1}. Thus we can assume $M$ is contained in $P(\Orbit(\omega))$ for some cycle $\omega$.

   Now take $Q=\vee\alpha^n(M)$. $Q$ commutes with the representation, by Proposition \ref{pr2.4}, and $M$ is cyclic for the subrepresentation of $\O_N$ on $Q\H$. The projection valued measure corresponding to this subrepresentation will be $P^r(A)=P(A)Q=QP(A)$ for all Borel subsets $A$ of $\K$. Clearly $P^r$ is atomic and supported on $\Orbit(\omega)$. With Theorem \ref{th2.6}, we have $P^r(\omega)\H\subset M$ so $QP(\omega)\subset M$.
   We can not have $QP(\omega)=0$, because then we apply $S_w,S_w^*$ for all finite words, and with Lemma \ref{lem2.7.1}, we obtain that $QP(\alpha)=0$ for all $\alpha\in\Orbit(\omega)$ and so $Q=0$, a contradiction. 
   Take a vector $v$ in $QP(\omega)$, then $v$ is in $P(\omega)\H\cap M$. 

   \end{proof}

   \begin{corollary}\label{cor4.18}
   If a purely atomic representation of $\O_N$ has a finite dimensional, cyclic, invariant space $M$ for all $S_i^*$, $i\in\bz_N$, then the support contains only periodic points, and for all cycles $\omega$ in the support, 
   $P(\omega)\H$ is contained in $M$.
   \end{corollary}
   
   \begin{proof}
   Follows directly from Theorem \ref{th2.6}.
   \end{proof}

   \begin{corollary}\label{cor4.19}
   Let $(S_i)_{i\in\bz_N}$ be a purely atomic representation of $\O_N$ with projection valued measure $P$ with support $\Omega$. If the representation has a finite dimensional cyclic $(S_i^*)$-invariant space, then there is a unique minimal finite dimensional cyclic $(S_i^*)$-invariant space:
   $$M=\oplus\{P(\omega)\H:\omega\mbox{ cycle in }\Omega\}.$$
   \end{corollary}
   
   \begin{proof}
   Corollary \ref{cor4.18} shows that every finite dimensional cyclic $(S_i^*)$-invariant space must contain $M$. From Remark \ref{rem3.5} we see that $M$ is cyclic and $(S_i^*)$-invariant. 
   \end{proof}

   \section{Examples}
   We already mentioned a host of areas where the representations of $\O_N$  play a central role, and in the present section we outline four:  wavelets, finitely correlated states, Hadamard triples, and Walsh bases. We show that, with our present theorems, we are able to advance conclusions contained in earlier papers on these subjects, see e.g., \cite{BJKW00, MR2362879,  MR1285568, MR2318495, DuJo06a, DuJo08a, MR2966145, MR2563494, MR1158756}.   
   
   \subsection{Wavelet representations}
   
   \begin{definition}\label{def4.1}
   Let $\bt:=\{z\in\bc : |z|=1\}$ with the Haar measure. Let $N\geq 2$. A family of $N$ functions $m_i$ in $L^2(\bt)$, $i\in\bz_N$, is called a {\it QMF system} if 
   \begin{equation}
\frac{1}{N}\sum_{w^N=z}m_i(w)\cj{m_j(w)}=\delta_{ij},\quad(i,j\in\bz_N).
\label{eq4.1.1}
\end{equation}
   
   Given a QMF system $(m_i)_{i\in\bz_N}$, one defines the operators $S_i$ on $L^2(\bt)$, by
   \begin{equation}
S_if(z)=m_i(z)f(z^N),\quad(f\in L^2(\bt), z\in\bt).
\label{eq4.1.2}
\end{equation}
   \end{definition}
  
  \begin{proposition}\label{pr4.2}\cite{MR2277210}
  The operators $(S_i)_{i\in\bz_N}$ in Definition \ref{def4.1} form a representation of $\O_N$ on $L^2(\bt)$ called the wavelet representation associated to the QMF system $(m_i)_{i\in\bz_N}$. The adjoints are given by the formula
  \begin{equation}
S_i^*f(z)=\frac{1}{N}\sum_{w^N=z}\cj{m_i(w)}f(w),\quad(f\in L^2(\bt),z\in\bt).
\label{eq4.2.1}
\end{equation}
  
  \end{proposition}

  \begin{theorem}\label{th4.3}
  Let $(S_i)_{i\in\bz_N}$ be the wavelet representation associated to a QMF system $(m_i)_{i\in\bz_N}$, where all the functions $m_i$ are trigonometric polynomials. Then the following statements are equivalent:
  
  \begin{enumerate}
	\item The representation has an atom. 
	\item There exist $i_0,\dots, i_{p-1}$ in $\bz_N$, $p\geq 1$, such that for all $j=0,\dots,p-1$, we have $m_{i_j}(z)=a_jz^{k_j}$ for some $a_j\in\bt$, $k_j\in\bz$ and $k_0+Nk_1+\dots N^{p-1}k_{p-1}$ is a multiple of $N^p-1$. 
\end{enumerate}
  
  \end{theorem}

  \begin{proof}
  Suppose the representation has an atom. By \cite[Proposition 3.1, Corollary 3.3]{BEJ00}, there exists a finite set $H$ in $\bz$ such that the space $M=\Span\{z^h : h\in H\}$ is cyclic and invariant.
  
 By Theorem \ref{th2.6}, there is a cycle word $\omega=\underline I$ with $P(\omega)\neq0$, and any vector $v$ in the atom $P(\omega)L^2(\bt)$ is contained also in $M$. The isometry $S_I$ has the form 
 $$S_If(z)=m_{i_0}(z)m_{i_1}(z^N)\dots m_{i_{p-1}}(z^{N^{p-1}})f(z^{N^p}),\quad (f\in L^2(\bt),z\in\bt).$$
 The fact that $v$ is in the atom $P(\omega)L^2(\bt)$ means that the isometry $S_I$ has a non-trivial unitary part in its Wold decomposition. Then, with \cite[Theorem 3.1]{BrJo97}, we obtain that for $$m(z):=m_{i_0}(z)m_{i_1}(z^N)\dots m_{i_{p-1}}(z^{N^{p-1}}),$$ we have $|m(z)|=1$ for all $z\in\bt$, the atom is 1-dimensional and $P(\omega)\H=\bc v$, where $v\in L^2(\br)$ satisfies 
 \begin{equation}
m(z)v(z^{N^p})=\lambda v(z),
\label{eq4.3.1}
\end{equation}
  for some $\lambda\in\bt$. 
 
 First, we show that $m(z)=az^l$ for some $a\in\bc$, $l\in \bz$. We know that $m$ is a trigonometric polynomial. Take $r\in\bz$ large enough so that $P(z):=z^rm(r)$ is a polynomial. We have $|P(z)|=1$ for $|z|=1$. Then, the function $R(z)=P(z)\cj{P(1/\cj z)}$ is an entire meromorphic function and $R(z)=1$ for $z\in\bt$. Then $R(z)=1$ for all $z\in\bc$ so $P(z)=\frac{1}{\cj{P(1/\cj z)}}$.  But since $P(z)$ is a polynomial, this implies that the only zero for $P(z)$ is $0$. Thus $m(z)=az^l$ for some $a\in\bc$, $l\in\bz$. Since the trig polynomials $m_{i_j}(z^{N^j})$ divide $m(z)$ it follows that $m_{i_j}(z)=a_jz^{k_j}$ for some $a_j\in\bc$ and $k_j\in\bz$. Moreover, the QMF conditions imply that $a_j\in\bt$. 
 
 Since the vector $v$ is in $M$, it is a trigonometric polynomial, $v=\sum_{j=e}^dv_jz^d$, and we can assume $v_d\neq 0$. Using \eqref{eq4.3.1} and equating the larges powers, we get $l+dN^p=d$ so $l=d(1-N^p)$.
 Also, the formula for $m$ shows that $l=k_0+Nk_1+\dots+N^{p-1}k_{p-1}$ and this implies (ii).
 
 For the converse let $l=k_0+Nk_1+\dots+N^{p-1}k_{p-1}=d(1-N^p)$. Then, with $m$ as above, and $v=z^d$, and $\lambda=a_{k_0}\dots a_{k_{p-1}}$, we have that \eqref{eq4.3.1} is satisfied. This means that $S_Iv=\lambda v$, iterating, we get that $S_I^kv=\lambda^kv$ so $v\in P(\omega)\H$.

  \end{proof}

  \begin{corollary}\label{cor4.4}
  If none of the trigonometric polynomials in a QMF system is of the form $az^k$, $a\in\bc$, $k\in\bz$, then the associated wavelet representation has no atoms. In particular, if the low-pass condition $m_0(1)=\sqrt{N}$ is satisfied then the wavelet representation has no atoms.
    \end{corollary}
  
  \begin{proof}
  This follows directly from Theorem \ref{th4.3}. If the low-pass condition is satisfied, then we also have $m_i(1)=0$ for all $i\neq 0$, so none of the functions can be of the form $az^k$. 
  \end{proof}
  
  \subsection{Finitely correlated states}

  \begin{theorem}\label{th5.1}\cite{BJKW00}
  Consider two representation $(S_i)_{i\in\bz_N}$, $(\tilde S_i)_{i\in\bz_N}$ of $\O_N$ on the Hilbert spaces $\H$, $\tilde\H$. Suppose each has a cyclic invariant (for the $S_i^*$, $\tilde S_i^*$ respectively) subspace $M$ and $\tilde M$ and let $V_i^*=S_i^*P_M$, $\tilde V_i^*=\tilde S_i^*P_{\tilde M}$. There is an isometric linear isomorphism between intertwiners $U:\H\rightarrow \tilde \H$, i.e., operators satisfying 
  $$US_i=\tilde S_iU,\quad (i\in\bz_N)$$
  and operators $V:M\rightarrow \tilde M$ such that 
  \begin{equation}
\sum_{i\in\bz_N}\tilde V_iVV_i^*=V;
\label{eq5.1.1}
\end{equation}
  given by the map $U\mapsto V=P_{\tilde M}UP_M$. 
  
  The operators $V_i$ satisfy the following relation (on $M$):
  \begin{equation}
\sum_{i\in\bz_N}V_iV_i^*=I
\label{eq5.1.2}
\end{equation}
  \end{theorem}

\begin{theorem}\label{th2.4}
Let $(S_i)_{i\in\bz_N}$ be a representation of $\O_N$ on a Hilbert space $\H$ and let $P$ be the associated projection valued measure. Suppose $M$ is a finite dimensional subspace of $\H$ invariant for all $S_i^*$, $i\in\bz_N$ and let $\{e_i: i=1,n\}$ be an orthonormal basis for $M$. Write the vectors $S_l^*e_i$ in this orthonormal basis
\begin{equation}
S_l^*e_i=\sum_{j=1}^n\alpha_{ji}^le_j.
\label{eq2.4.1}
\end{equation}
Define the matrices $Z_l=(\alpha_{ij}^l)_{i,j=1}^n$. Define, for each Borel subset $B$ of $\K$, the matrix
\begin{equation}
\mu_M(B):=(\ip{P(B)e_i}{e_j})_{i,j=1}^n.
\label{eq2.4.2}
\end{equation}
Then 
\begin{equation}
\sum_{l\in\bz_N}Z_l^*Z_l=I.
\label{eq2.4.3}
\end{equation}

For all Borel sets $B$ in $\K$
\begin{equation}
\mu_M(B)=\sum_{l\in\bz_N}Z_l^*\mu_M(\sigma_l^{-1}(B))Z_l.
\label{eq2.4.3.0}
\end{equation}
For all finite words $i_1\dots i_m$ in $\bz_N^m$, 
\begin{equation}
\mu_M(B)=Z_I^*Z_I,
\label{eq2.4.4}
\end{equation}
with the usual notation $Z_I:=Z_{i_1}\dots Z_{i_m}$. 
\end{theorem}

\begin{proof}
Apply $S_l$ to \eqref{eq2.4.1}, sum over $l\in\bz_N$ and use the Cuntz relation. This yields \eqref{eq2.4.3}.

For a Borel set $B$ in $\K$, we have 
$$\ip{P(B)e_i}{e_j}=\sum_{l\in\bz_N}\ip{P(B\cap\C(l))e_i}{e_j}=\sum_{l\in\bz_N}\ip{P(\sigma_l\sigma_l^{-1}(B))e_i}{e_j}$$$$=\sum_{l\in\bz_N}\ip{S_lP(\sigma_l^{-1}(B))S_l^*e_i}{e_j}
=\sum_{l\in\bz_N}\ip{P(\sigma_l^{-1}(B))S_l^*e_i}{S_l^*e_j}$$$$=\sum_{l\in\bz_N}\sum_{i',j'=1}^n\alpha_{i'i}^l\cj{\alpha_{j'j}^l}\ip{P(\sigma_l^{-1}(B))e_{i'}}{e_{j'}}.$$
This implies \eqref{eq2.4.3.0}. The relation \eqref{eq2.4.4} follows from \eqref{eq2.4.3.0}, taking $B$ to be a cylinder and using induction.

\end{proof}

%
%

%
%
  
  \begin{definition}\label{def5.2}
  A representation of $\O_N$ on a Hilbert space $\H$ is called {\it generic} if there exists a cyclic vector $\psi$ in $\H$ with $\|\psi\|=1$ and constants $z_i\in\bc$, $i\in\bz_N$ such that 
  $$S_i^*\psi=z_i\psi,\quad(i\in\bz_N).$$
  Note that in this case, from the Cuntz relation, we obtain that 
  $$\sum_{i\in\bz_N}|z_i|^2=1.$$
  
  \end{definition}
  
  \begin{proposition}\label{pr5.3}
     The following statements hold:
  \begin{enumerate}
	\item Every generic representation of $\O_N$ is irreducible. 
	\item Two generic representations of $\O_N$, with constants $z=(z_i)_{i\in\bz_N}$ and $\tilde z=(\tilde z_i)_{i\in\bz_N}$ are equivalent if the corresponding vectors $z$ and $\tilde z$ are proportional, $z=\lambda z$, $|\lambda|=1$; they are disjoint otherwise. 
	\item For a generic representation of $\O_N$ with vector $\psi$ and constants $(z_i)_{i\in\bz_N}$, let $P$ be the associated projection valued measure and let $\mu_\psi$ be the measure on $\K_N$ defined by
	$$\mu_\psi(B)=\ip{P(B)\psi}{\psi},\quad (B\in \B(\K_N)).$$
	Then, the measure $\mu_\psi$ satisfies the following invariance equation
	\begin{equation}
\mu_\psi(B)=\sum_{l\in\bz_N}|z_l|^2\mu_\psi(\sigma_l^{-1}(B)),\quad(B\in\B(\K_N)).
\label{eq5.3.1}
\end{equation}
\item A generic representation of $\O_N$ has an atom if and only if $|z_i|=1$ for one of $i\in\bz_N$. In this case, the representation is purely atomic, supported on $\{w\underline i : w\mbox{ finite word }\}$, and it is permutative iff $z_i=1$. 

\end{enumerate}
  
  \end{proposition}
  
  \begin{proof}
  The first two statements follow from Theorem \ref{th5.1}: the space $M$ spanned by the vector $\psi$ is cyclic, one-dimensional and invariant for the $S_i^*$'s. The commutant of the representation is in 1-1 correspondence with numbers $V$ such that $\sum_{i\in\bz_N}|z_i|^2V=V$, but those are all the numbers in $\bc$. 
  
  For two such generic representations, the intertwiners are in 1-1 correspondence with numbers $V$ such that
  $$V\sum_{i\in\bz_N}z_i\cj{\tilde z_i}=V.$$
  
  To have non-zero solutions, we must have that $\sum _i z_i\cj{\tilde z_i}=1$ and this means we have equality in the Schwarz inequality so the vectors $z$ and $\tilde z$ are proportional.

  (iii) follows from Theorem \ref{th2.4}. If the representation has an atom then, by (i) and Proposition \ref{pr2.10} it is purely atomic. By Corollary \ref{cor4.18} every atom has to be periodic and every atom which is a cycle $\underline I$ has to be contained in the span of $\psi$. Then $S_I^*$ has to be unitary on $\psi$ so $|z_{i_1}\dots z_{i_p}|=1$. Since $\sum_i|z_i|^2=1$ this implies that all the $i_k$ are equal to some $i\in\bz_N$ and $|z_i|=1$. Therefore $|z_j|=0$ for $j\neq 0$. The representation is supported on $\{w\underline i : w\mbox{ finite word }\}$. 
  
  If $|z_i|=1$ then $\underline i$ is an atom and the representation is supported on its orbit (by Proposition \ref{pr2.10}). 
  
  If $z_i=1$ then $\{S_{w\underline 1}\psi : w\mbox{ finite word }\}$ is a permutative orthonormal basis. Conversely, if the representation is permutative, then by Theorem \ref{th1.5}, since the cycle $\underline i$ is the span of $\psi$, we must have $S_i^*\psi=\psi$ so $z_i=1$. 
  \end{proof}

\begin{remark}\label{rem2.6}
It follows that the measure $\mu_\psi$ is the infinite product measure corresponding to the weights form \eqref{eq5.3.1} on the letters in $\bz_N$. As a result, pairs of distinct weights on $\bz_N$ yield pairs of mutually singular measures $\mu_\psi$ on $\K_N$ (by Kakutani).

\end{remark}

\begin{example}
Consider the wavelet representation associated to the Haar filters $m_0(z)=\frac{1+z}{\sqrt2}$, $m_1(z)=\frac{1-z}{\sqrt2}$, $z\in\bt$, $N=2$. 

Then 
$$S_0^*1(z)=\frac12\sum_{w^2=z}\frac{1+\cj{w}}{\sqrt2}=\frac1{\sqrt2},\quad S_1^*1(z)=\frac12\sum_{w^2=z}\frac{1-\cj{w}}{\sqrt2}=\frac1{\sqrt2}.$$
Then, by Corollary \ref{pr5.3}, the measure $\mu_1$ satisfies the invariance equation
$$\mu_1(B)=\frac12(\mu_1(\sigma_0^{-1}(B))+\mu_1(\sigma_1^{-1}(B))),\quad(B\in\B(\K_2)).$$
But this means that $\mu_1$ is the Haaar measure on $\K_2$.

It follows that this measure pulls back to become the Lebesgue measure on the unit interval, given as an iterated function system measure. 
\end{example}

\begin{example}
Consider the wavelet representation associated to the Cantor filters $m_0(z)=\frac{1+z^2}{\sqrt2}$, $m_1(z)=z$, $m_2(z)=\frac{1-z^2}{\sqrt{2}}$, $z\in\bt$, $N=3$. See \cite{DuJo06a}. Then 
$$S_0^*1(z)=\frac13\sum_{w^3=z}\frac{1+\cj w^2}{\sqrt 2}=\frac1{\sqrt2},\quad S_1^*1(z)=\frac13\sum_{w^3=z}\cj w=0,\quad S_2^*1(z)=\frac13\sum_{w^3=z}\frac{1-\cj w^2}{\sqrt2}=\frac1{\sqrt2}.$$
Then, by Corollary \ref{pr5.3}, the measure $\mu_1$ satisfies the invariance equation
$$\mu_1(B)=\frac12(\mu_1(\sigma_0^{-1}(B))+\mu_1(\sigma_2^{-1}(B))),\quad(B\in\B(\K_3)).$$
But this means that $\mu_1$ is the middle-third-Cantor measure on $\B(\K_3)$.

It follows that this measure pulls back to become the middle-third  Cantor measure with its support on the middle-third Cantor set, given as an iterated function system measure.

\end{example}

\begin{definition}\label{def7.10}
Let $(V_i)_{i\in\bz_N}$ be some operators on a Hilbert space $M$, satisfying the relation
\begin{equation}
\sum_{i\in\bz_N}V_iV_i^*=I.
\label{eq7.10.1}
\end{equation}
We say that the operators $(V_i)_{i\in\bz_N}$ are {\it block permutative} if there exists a decomposition of $M$ into orthogonal subspaces $M=\oplus_{j\in J}M_{j}$ such that for every $j\in J$ and every $i\in\bz_N$, the operator $V_i^*$ restricted to $M_j$ is either unitary onto some $M_{j'}$ with $j'\in J$ or constant 0. 

We say that the operators $(V_i)_{i\in\bz_N}$ are {\it permutative} if there exists an orthonormal basis $(e_j)_{j\in J}$ for $M$ such that for all $j\in J$ and all $i\in \bz_N$, either $V_i^*e_j=e_{j'}$ for some $j'\in J$, or $V_i^*e_j=0$ 
\end{definition}

\begin{remark}\label{rem7.11}
If $V_i^*$ is unitary from $M_j$ to $M_{j'}$ as above then, because of the relation \eqref{eq7.10.1}, for all $i'\neq i$, $V_{i'}^*$ is zero on $M_j$. Similarly, if $V_i^*e_j=e_{j'}$, then $V_{i'}^*e_j=0$ for all $i'\neq i$. 

\end{remark}

\begin{theorem}\label{th7.12}
Let $(S_i)_{i\in\bz_N}$ be a representation of $\O_N$ on a Hilbert space $\H$ and let $M$ be minimal finite dimensional cyclic $(S_i^*)$-invariant subspace. Let $V_i^*=S_i^*P_M$. Then
\begin{enumerate}
	\item The representation is purely atomic if and only if $(V_i)_{i\in\bz_N}$ is block permutative.
	\item The representation is permutative if and only if $(V_i)_{\in\bz_N}$ is permutative.
\end{enumerate}

\end{theorem}

\begin{proof}
If the representation is purely atomic then, by Corollary \ref{cor4.19}, 
$$M=\oplus\{P(\omega)\H : \omega\mbox{ cycle in the support }\}.$$
Let $M_\omega=P(\omega)\H$. Then for all $\omega$ cycle in the support and all $i\in\bz_N$, by Lemma \ref{lem2.7.1}, we have that, on $M_\omega$, $V_i^*=S_i^*P_M$ is either unitary onto $M_{\sigma(\omega)}$ or zero.

For the converse, for each $j\in J$ there exists a unique $\omega_1$ such that $V_{\omega_1}^*$, so $S_{\omega_1}^*$, is unitary from $M_j$ to $M_{j'}$. Inductively, we can construct $\omega=\omega_1\omega_2\dots$ such that  $S_{\omega|k}^*$ is unitary on $M_j$. We identify $j$ with $\omega$. Note that if $v\in M_\omega$ then $S_{\omega|k}S_{\omega|k}^*v=v$ so $v\in P(\omega|k)\H$ for all $k\geq 1$, which means that $v\in P(\omega)\H$. Thus, $M_\omega\subset P(\omega)\H$. 

Also, if $v\in M_\omega$ and $I$ is some finite word, then 
$$S_{I\omega|k}S_{I\omega|k}^*(S_Iv)=S_IS_{\omega|k}S_{\omega|k}^*v=S_Iv,$$
therefore $S_Iv\in P(I\omega|k)\H$ for all $k\geq 1$. This means that $S_Iv\in P(I\omega)\H$. 

Since $M$ is cyclic, we have that $$\Span\{S_Iv : v\in \oplus_\omega M_\omega, I\mbox{ finite word}\}=\H.$$
So $\oplus P(I\omega)=I$ and the representation is purely atomic.

For (ii), if the representation is permutative, then $M$ contains all cyclic atoms, by Corollary \ref{cor4.19}, and these atoms contain all the vectors in the permutative orthonormal basis that have a cyclic encoding, by Theorem \ref{th1.2}. These vectors make $(V_i)_{i\in \bz_N}$ permutative. 

For the converse, if $(V_i)_{i\in\bz_N}$ is permutative then it is also block permutative with blocks corresponding the one dimesional spans of each vector $e_j$. 

Thus, by (i), the representation is purely atomic and supported on the orbits of the cyclic atoms that are contained in $M$ and the orthonormal basis for $M$, $\{e_j: j\in J\}$ splits into several orthonormal bases, one for each cyclic atom contained in $M$, satisfying the condition (ii) of Theorem \ref{th1.5}. Therefore the representation is permutative. 

\end{proof}

\subsection{Representations associated to Hadamard triples} The next example involves representations of $\O_N$ associated to orthonormal Fourier bases on fractal measures. They were studied in \cite{DuJo12}. 

\begin{definition}\label{def8.1}
Let $R\geq2$ be an integer and let $B$ and $L$ be two finite subsets of $\bz$ with $0\in B,L$ and having the same cardinality $|B|=|L|=:N$. We say that $(R,B,L)$ forms a Hadamard triple if the matrix
\begin{equation}
\frac{1}{\sqrt N}\left(e^{2\pi i\frac1Rb\cdot l}\right)_{b\in B,l\in L}
\label{eq8.1.1}
\end{equation}
is unitary. 
\end{definition}

\begin{definition}\label{def8.2}
Let $(R,B,L)$ be a Hadamard triple. Define the affine maps
$$\tau_b(x)=R^{-1}(x+b),\quad(x\in\br, b\in B).$$

By \cite{Hut81}, there exists a unique compact set $X_B$ called {\it the attractor} of the IFS $(\tau_b)_{b\in B}$ such that
\begin{equation}
X_B=\bigcup_{b\in B}\tau_b(X_B).
\label{eq1.2}
\end{equation}
In our case, it can be written explicitly
\begin{equation}
X_B=\left\{\sum_{k=1}^\infty R^{-k}b_k : b_k\in B\mbox{ for all }k\geq1\right\}.
\label{eq1.3}
\end{equation}

There exists a unique Borel probability measure $\mu_B$ such that 
\begin{equation}
\mu_B(E)=\frac{1}{N}\sum_{b\in B}\mu_B(\tau_b^{-1}(E))\mbox{ for all Borel subsets }E\mbox{ of }\br^d.
\label{eq1.4}
\end{equation}
Equivalently 
\begin{equation}
\int f\,d\mu_B=\frac1N\sum_{b \in B}\int f\circ \tau_b\,d\mu_B\mbox{ for all bounded Borel function }f\mbox{ on }\br^d.
\label{eq1.5}
\end{equation}
The measure $\mu_B$ is called the {\it invariant measure} of the IFS $(\tau_b)_{b\in B}$. It is supported on $X_B$.

The measure $\mu_B$ {\it has no overlap}, (see \cite[Theorem 1.10]{DHS12}), i.e., 
\begin{equation}
\mu_B(\tau_b(X_B)\cap \tau_{b'}(X_B))=0,\mbox{ for all }b\neq b'\in B.
\label{eq1.6}
\end{equation}

Since the measure $\mu_B$ has no overlap, then one can define the map $\R:X_B\rightarrow X_B$
\begin{equation}
\R(x)=Rx-b,\mbox{ if }x\in\tau_b(X_B).
\label{eq1.7}
\end{equation}
The map $\R$ is well defined $\mu_B$-a.e. on $X_B$.

For $\lambda\in \br$, denote by $e_\lambda(x)=e^{2\pi i\lambda x}$, $x\in\br$.

We define the function \begin{equation}
m_B(x)=\frac{1}{\sqrt{N}}\sum_{b\in B}e^{2\pi i bx},\quad(x\in\br).
\label{eq1.8}
\end{equation}

A set $\{x_0,\dots,x_{p-1}\}$ of points in $\br$ is called an extreme cycle if there exist $l_0,\dots,l_{p-1}\in L$ such that 
$$\frac1R(x_0+l_0)=x_1,\dots,\frac1R(x_{p-2}+l_{p-2})=x_{p-1},\frac1R(x_{p-1}+l_{p-1})=x_0,$$
and 
$$|m_B(x_i)|=\sqrt{N},\quad(i=0,\dots,p-1).$$

For an extreme cycle $C$, let $\Lambda(C)$ is the smallest set such that $\Lambda(C)$ contains $-C$ and such that $R\Lambda+L\subset \Lambda$. Let $\Lambda$ be the union of all $\Lambda(C)$ with $C$ extreme cycle. 

Identify $L$ with $\bz_N$ and define the map $E:\Lambda\rightarrow\K_N$ by 
$E(\lambda)=l_1l_2\dots$, where $\lambda=l_1+R\lambda_1$, $l_1\in L$, $\lambda_1\in\Lambda$, $\lambda_1=l_2+R\lambda_2,\dots$.

For any point $x_0$ in an extreme cycle as before, $E(-x_0)=\underline{l_{p-1}\dots l_0}$.

Define the operators $S_l$, $l\in L$ on $L^2(\mu_B)$ by 
\begin{equation}
S_lf=e_lf\circ\R,\quad(l\in L).
\label{eq8.2.2}
\end{equation}

\end{definition}

\begin{theorem}\label{opr8.3}
The operators $(S_l)_{l\in L}$ form a permutative representation of $\O_N$. A permutative orthonormal basis is $\{e_\lambda :\lambda\in \Lambda\}$. The maps $\sigma_l$ on $\Lambda$ are given by $\sigma_l(\lambda)=l+R\lambda$, $l\in L$, $\lambda\in\Lambda$. The encoding map is $E$ and it is injective. The decomposition of the representation into inequivalent irreducible representations is given by the spaces $\H(C)=\{e_\lambda :\lambda\in\Lambda(C)\}$ for all the extreme cycles $C$. The associated projected valued measures $P$ are purely atomic and 
$P(E(\lambda))$ is the projection onto the function $e_\lambda$. If $x_0$ is a point in an extreme cycle $C$, then $E(-x_0)$ is cyclic and $E(\Lambda(C))=\Orbit(E(-x_0))$. 

Every finite dimensional subspace $M$ which is invariant under all $S_l^*$, $l\in L$ must contain functions $e_{-c}$ for all $c$ points in one of the extreme cycles $C$. If in addition $M$ is cyclic then $M$ contains all such functions for all extreme cycles $C$. 
\end{theorem}

\begin{proof}
Many details are contained in \cite{DuJo12}, but we will sketch some of them and also use our results here. We have
$$S_le_\lambda=e_{l+R\lambda},\quad(\lambda\in\Lambda, l\in L)$$
and this proves that the maps $\sigma_l$ and the encoding $E$ are as given. The fact that $E$ is injective follows from the fact that the inverses of the maps $\sigma_l$ are strictly contracting. The completeness of the basis $\{e_\lambda :\lambda\in\Lambda\}$ is proved in \cite{DuJo12}, see also \cite{DJ06,DPS13}. The decomposition into irreducibles is presented in \cite{DuJo12}, but it also follows easily from our results: since the representation is permutative, it follows that it is purely atomic and $P(E(\lambda))$ is the projection onto $e_\lambda$, by Theorem \ref{th1.2}. The orbit of $E(-c)$, with $c$ a point in an extreme cycle, is $E(\Lambda(C))$. By Proposition \ref{pr3.6}, each such orbit gives a subrepresentation and, by Proposition \ref{pr2.11}, these subrepresentations are irreducible and inequivalent. The statements about $S_l^*$-invariant spaces follow from Theorem \ref{th2.6} and Corollary \ref{cor1.4}.

\end{proof}

\begin{example}\label{ex8.4}
The classical Fourier bases fit into this context. Let $R=2$, $B=L=\{0,1\}$. Then the measure $\mu_B$ is the Lebesgue measure on $[0,1]$. There are two extreme cycles: $\{0\}$ with $E(0)=\underline 1$ and $\{1\}$ with $E(-1)=\underline 1$. The sets $\Lambda(0)=\{n\in\bz : n\geq 0\}$ and $\Lambda(1)=\{n\in\bz : n<0\}$. The encoding map $E$ associates to an integer its the base 2 expansion; for non-negative numbers the expansion ends in $\underline 0$ and, for negative numbers, the expansion ends in $\underline1$. 
The isometries are
$$S_0f(x)=f(2x\mod1),\quad S_1(x)=e^{2\pi ix}f(2x\mod1),\quad(x\in\br, f\in L^2[0,1]).$$
This representation of $\O_2$ decomposes into two inequivalent irreducible representations on $\H(0)=H^2$, the Hardy space and on $\H(1)=\Span\{e_n : n<0\}$.

\end{example}

\subsection{Representations associated to Walsh bases}
Recall some facts from \cite{DPS13}. Let $N\geq 2$ and let $A$ be an $N\times N$ unitary matrix with constant first row $\frac{1}{\sqrt N}$. Let $\R(x)=Nx\mod 1$ on $[0,1]$. Define the QMF system
$$m_i(x)=\sqrt{N}\sum_{j=0}^{N-1}a_{ij}\chi_{[j/N,(j+1)/N)}(x).$$
Define the operators on $L^2[0,1]$
$$S_if=m_i\cdot(f\circ \R),\quad(f\in L^2[0,1], i\in\bz_N).$$

\begin{theorem}\label{th9.1}
The operators $(S_i)_{i\in\bz_N}$ form a permutative representation of $\O_N$ with permutative basis
\begin{equation}
\{S_w1: w\mbox{ finite word, either empty or ending in 1}\}.
\label{eq9.1.1}
\end{equation}

The encoding map $E$ associated to a word $w$ as in \eqref{eq9.1.1} is $E(w)=w\underline0$ and it is injective. The maps $\sigma_l$ are $\sigma_0(\ty)=\ty$, $\sigma_l(\ty)=l$ for $l\neq 0$ and $\sigma_l(w)=lw$ for $w\neq \ty$. 
The representation is irreducible, purely atomic, supported on the set of finite words ending in $\underline0$. Every finite dimensional subspace $M$ which is invariant for all $S_l^*$, $l\in\bz_N$ must contain $1$. 

\end{theorem}

  \begin{proof}
  The fact that these operators form a representation of $\O_N$ and that \eqref{eq9.1.1} gives an orthonormal basis, is proved in \cite{DPS13}. The formulas for $\sigma_l$ and $E$ are obvious. Since $S_01=1$ all encodings end in $\underline0$. The projection valued measure $P$ is supported on the words that end in $\underline 0$ and this is the orbit of $\underline 0$, $P(w\underline 0)$ is the projection onto $S_w1$, by Theorem \ref{th1.2}, so it is one-dimensional. Therefore the representation is irreducible, by Proposition \ref{pr2.11}.

  The statement about $S_l^*$-invariant spaces follows from Corollary \ref{cor1.4}.
  \end{proof}
  
  \begin{remark}\label{rem9.2}
  If $N=2$ and
  $$A=\frac{1}{\sqrt2}\begin{pmatrix}
	1&1\\1&-1
\end{pmatrix}
  $$
 then one gets exactly the classical Walsh basis on $L^2[0,1]$. 
 
 For $N=2$ and any unitary matrix $A$ as above, the corresponding representation of $\O_2$ is permutative with 1-dimensional atoms and supported on the words that end in $\underline 0$. The same is true for the subrepresentation defined in Example \ref{ex8.4} defined on the Hardy space $H^2$. Therefore, by Corollary \ref{cor2.14}, all these representations are equivalent. 
  \end{remark}

  \begin{acknowledgements}
This work was partially supported by a grant from the Simons Foundation (\#228539 to Dorin Dutkay).
\end{acknowledgements}

\bibliographystyle{alpha}	
\bibliography{eframes}

\end{document}